\documentclass[leqno,11pt]{article}

\usepackage[utf8]{inputenc}
\usepackage[T1]{fontenc}
\usepackage[english]{babel}

\usepackage[intlimits]{amsmath}
\usepackage{amssymb, amsfonts}
\usepackage{amsthm}

\usepackage{mathabx}

\usepackage{MnSymbol}
\usepackage{mathbbol}
\usepackage{bbm}
\usepackage{mathrsfs}

\usepackage[all]{xy}

\usepackage{graphicx}
\usepackage{color}

\numberwithin{equation}{section}

\newtheorem{theoremcounter}{theoremcounter}[section]
\newtheorem{thmstarcounter}{thmstarcounter}

\newtheorem{conjecture}[theoremcounter]{Conjecture}
\newtheorem*{conjstar}{Conjecture}
\newtheorem{corollary}[theoremcounter]{Corollary}
\newtheorem{lemma}[theoremcounter]{Lemma}

\newtheorem{proposition}[theoremcounter]{Proposition}
\newtheorem{theorem}[theoremcounter]{Theorem}
\newtheorem{thmstar}[thmstarcounter]{Theorem}

\theoremstyle{definition}

\newtheorem{notation}[theoremcounter]{Notation}

\newtheorem{remark}[theoremcounter]{Remark}

\newcommand{\cA}{\ensuremath{\mathcal{A}}}

\newcommand{\cF}{\ensuremath{\mathcal{F}}}

\newcommand{\cH}{\ensuremath{\mathcal{H}}}

\newcommand{\cO}{\ensuremath{\mathcal{O}}}

\newcommand{\cR}{\ensuremath{\mathcal{R}}}

\newcommand{\cZ}{\ensuremath{\mathcal{Z}}}

\newcommand{\rE}{\ensuremath{\mathrm{E}}}

\newcommand{\rL}{\ensuremath{\mathrm{L}}}

\newcommand{\rS}{\ensuremath{\mathrm{S}}}

\newcommand{\rmd}{\ensuremath{\mathrm{d}}}

\newcommand{\rmf}{\ensuremath{\mathrm{f}}}

\newcommand{\ol}{\overline}

\newcommand{\amid}{\ensuremath{\, | \,}}

\newcommand{\eqstop}{\ensuremath{\, \text{.}}}
\newcommand{\eqcomma}{\ensuremath{\, \text{,}}}

\newcommand{\NN}{\ensuremath{\mathbb{N}}}
\newcommand{\ZZ}{\ensuremath{\mathbb{Z}}}
\newcommand{\QQ}{\ensuremath{\mathbb{Q}}}
\newcommand{\RR}{\ensuremath{\mathbb{R}}}
\newcommand{\CC}{\ensuremath{\mathbb{C}}}

\newcommand{\id}{\ensuremath{\mathrm{id}}}
\newcommand{\ra}{\ensuremath{\rightarrow}}
\newcommand{\lra}{\ensuremath{\longrightarrow}}

\newcommand{\Cmat}[1]{\ensuremath{\mathrm{M}_{#1}(\CC)}}

\newcommand{\Tr}{\ensuremath{\mathop{\mathrm{Tr}}}}

\newcommand{\ot}{\ensuremath{\otimes}}

\newcommand{\Wstar}{\ensuremath{\text{W}^*}}

\newcommand{\bo}{\ensuremath{\mathscr{B}}}

\newcommand{\twonorm}{\ensuremath{\| \phantom{x} \|_2}}

\newcommand{\supp}{\ensuremath{\mathop{\mathrm{supp}}}}

\newcommand{\vnt}{\ensuremath{\overline{\otimes}}}
\newcommand{\bim}[3]{\mathord{\raisebox{-0.4ex}[0ex][0ex]{\scriptsize $#1$}{#2}\hspace{-0.2ex}\raisebox{-0.4ex}[0ex][0ex]{\scriptsize $#3$}}}
\newcommand{\lmod}[2]{\mathord{\raisebox{-0.4ex}[0ex][0ex]{\scriptsize $#1$}{#2}}}
\newcommand{\rmod}[2]{\mathord{{#1}\raisebox{-0.4ex}[0ex][0ex]{\scriptsize $#2$}}}

\newcommand{\Linfty}{\ensuremath{\mathrm{L}^\infty}}
\newcommand{\Ltwo}{\ensuremath{\mathrm{L}^2}}

\newcommand{\ltwo}{\ensuremath{\ell^2}}

\newcommand{\grpaction}[1]{\ensuremath{\stackrel{#1}{\curvearrowright}}}

\newcommand{\freegrp}[1]{\ensuremath{\mathbb{F}}_{#1}}

\DeclareFontFamily{OT1}{pzc}{}
\DeclareFontShape{OT1}{pzc}{m}{it}{<-> s * [1.10] pzcmi7t}{}
\DeclareMathAlphabet{\mathpzc}{OT1}{pzc}{m}{it}

\newcommand{\Norm}{\ensuremath{\mathcal{N}}}
\newcommand{\QN}{\ensuremath{\mathrm{QN}}}

\renewcommand{\vnt}{\ot}

\begin{document}
\begin{center}
\textbf{\LARGE On the classification of free Bogoljubov crossed product von Neumann algebras by the integers} \\

\bigskip

{by Sven Raum \footnote{Supported by KU Leuven BOF research grant OT/08/032}}
\end{center}

\begin{center}
  \textbf{Abstract.} We consider crossed product von Neumann algebras arising from free Bogoljubov actions of $\ZZ$.  We describe several presentations of them as amalgamated free products and cocycle crossed products and give a criterion for factoriality.  A number of isomorphism results for free Bogoljubov crossed products are proved, focusing on those arising from almost periodic representations.  We complement our isomorphism results by rigidity results yielding non-isomorphic free Bogoljubov crossed products and by a partial characterisation of strong solidity of a free Bogoljubov crossed products in terms of properties of the orthogonal representation from which it is constructed
\end{center}

\section{Introduction}
\label{sec:introduction}

With an orthogonal representation $(H, \pi)$ of a discrete group $G$, Voiculescu's free Gaussian functor associates an action of $G$ on the free group factor $\Gamma(H)'' \cong \rL \freegrp{\dim H}$ (see Section \ref{sec:preliminaries:representations} and \cite[Section 2.6]{voiculescudykemanica92}).  An action arising this way is called a \emph{free Bogoljubov action} of $G$.  The associated free Bogoljubov crossed product von Neumann algebras $\Gamma(H)'' \rtimes G$, also denoted by $\Gamma(H, G, \pi)''$, were studied by several authors \cite{shlyakhtenko99, houdayershlyakhtenko11, houdayer12, houdayer12-structure}.  Note that in \cite[Section 7]{shlyakhtenko99} free Bogoljubov crossed products with $\ZZ$ appear under the name of \emph{free Krieger algebras} (see also \cite[Section 3]{shlyakhtenko98-applications} and \cite[Section 6]{houdayershlyakhtenko11}).  The classification of free Bogoljubov crossed products is especially interesting because of their close relation to \emph{free Araki-Woods factors} \cite{shlyakhtenko97, shlyakhtenko99}.  In the context of the complete classification of free Araki-Woods factors associated with almost periodic orthogonal representations of $\RR$ \cite[Theorem 6.6]{shlyakhtenko97}, already the classification of the corresponding class of free Bogoljubov crossed products becomes an attractive problem.

Popa initiated his \emph{deformation/rigidity theory} in 2001 \cite{popa06-non-commutative-bernoulli-shifts, popa06, popa06_2, popa06_3, popa08-spectral-gap}.  During the past decade this theory enabled him to prove a large number of non-isomorphism results for von Neumann algebras and to calculate many of their invariants.  In particular, he obtained the first rigidity results for \emph{group measure space II$_1$ factors} in \cite{popa06_2, popa06_3}. Moreover, he obtained the first calculations of \emph{fundamental groups} not equal to $\RR_{>0}$ in \cite{popa06} and of outer automorphisms groups in \cite{ioanapetersonpopa08}.  Further developments in the deformation/rigidity theory led Ozawa and Popa to the discovery of \emph{II$_1$ factors with a unique Cartan subalgebra} in \cite{ozawapopa10-cartan1, ozawapopa10-cartan2}. Also \emph{\Wstar-superrigidity} theorems for group von Neumann algebras \cite{ioanapopavaes10,berbecvaes12} and group measure space II$_1$ factors \cite{popavaes10-superrigidity, popavaes11_2, popavaes12, ioana12} were proved by means of deformation/rigidity techniques.  In the context of free Bogoljubov actions Popa's techniques were applied too.  In \cite[Section 6]{popa06-non-commutative-bernoulli-shifts}, Popa introduced the \emph{free malleable deformation} of free Bogoljubov crossed products.  This lead in \cite{houdayer10} and, using the work of Ozawa-Popa, in \cite{houdayershlyakhtenko11, houdayerricard11-araki-woods, houdayer12-structure} to several structural results and rigidity theorems for free Araki-Woods factors and free Bogoljubov crossed products.  We use the main result of \cite{houdayershlyakhtenko11} in order to obtain certain non-isomorphism results for free Bogoljubov crossed products.

In the cause of the deformation/rigidity theory, absence of Cartan algebras and primeness were studied too.  The latter means that a given II$_1$ factor has no decomposition as a tensor product of two II$_1$ factors.  Ozawa introduced in \cite{ozawa04-solid} the notion of \emph{solid} II$_1$ factors, that is II$_1$ factors $M$ such that for all diffuse von Neumann subalgebras $A \subset M$ the relative commutant $A' \cap M$ is amenable.  In \cite{popa07-on-ozawa}, Popa used his deformation/rigidity techniques in order to prove solidity of the free group factors, leading to the discovery of \emph{strongly solid} II$_1$ factors in \cite{ozawapopa10-cartan1, ozawapopa10-cartan2}.  A II$_1$ factor $M$ is strongly solid if for all amenable, diffuse von Neumann subalgebras $A \subset M$, its normaliser $\Norm_M(A)''$ is amenable too.  We extend the results of \cite{houdayershlyakhtenko11} on strong solidity of certain free Bogoljubov crossed products and point out a class of non-solid free Bogoljubov crossed products.

Opposed to non-isomorphism results obtained in Popa's deformation/rigidity theory, there are two known sources of isomorphism results for von Neumann algebras. First, the \emph{classification of injective von Neumann algebras} by Connes \cite{connes76} shows that all group measure space II$_1$ factors $\Linfty(X) \rtimes G$ associated with free, ergodic, probability measure preserving actions $G \grpaction{} X$ are isomorphic to the hyperfinite II$_1$ factor $R$. By \cite{ornsteinweiss80, connesfeldmanweiss81}, if $H \grpaction{} Y$ is another free, ergodic, probability measure preserving action of an amenable group, then these actions are \emph{orbit equivalent}, meaning that there is a probability measure preserving isomorphism $\Delta: X \ra Y$ such that $\Delta(G \cdot x) = H \cdot \Delta(x)$ for almost every $x \in X$.  By a result of Singer \cite{singer55}, this means that there is an isomorphism $\Linfty(X) \rtimes G \cong \Linfty(Y) \rtimes G$ sending $\Linfty(X)$ to $\Linfty(Y)$.

The second source of unexpected isomorphism results for von Neumann algebras is \emph{free probability theory} as it was initiated by Voiculescu \cite{voiculescu85}.  We employ two branches of free probability theory.  On the one hand, we use the work of Dykema on \emph{interpolated free group factors} and \emph{amalgamated free products}.  Interpolated free group factors were independently introduced by Dykema \cite{dykema94-interpolated} and R\u{a}dulescu \cite{radulescu94-interpolated}.  If $M$ is a II$_1$ factor, the \emph{amplification of $M$ by $t$} is $M^t = p(\Cmat{n} \ot M)p$, where $p \in \Cmat{n} \ot M$ is a projection of non-normalised trace $\Tr \ot \tau(p) = t$ .  It does not depend on the specific choice of $n$ and $p$.  The interpolated free group factors can be defined by
\[
  \rL \freegrp{r}
  =
  (\rL \freegrp{n})^t
  \eqcomma
  \text{ where } r  = 1 + \frac{n-1}{t^2} \eqcomma \text{ for some } t > 1 \text{ and } n \in \NN_{\geq 2}
  \eqstop
\]
Dykema's first result on free products of von Neumann algebras in \cite{dykema93-matrix-model} says that $\rL(\freegrp{n}) * R \cong \rL(\freegrp{n + 1})$ for any natural number $n$.  He developed his techniques in \cite{dykema94-interpolated, dykema93-hyperfinite, dykema95-multi-matrix, dykema11-finite-vnalg} arriving in \cite{dykemaredelmeier11} at a description of arbitrary amalgamated free products $A *_D B$ with respect to trace-preserving conditional expectations, where $A$ and $B$ are tracial direct sums of hyperfinite von Neumann algebras and interpolated free group factors and the amalgam $D$ is finite dimensional.

We combine the work of Dykema with a result on factoriality of certain amalgamated free products.  The first such results for proper amalgamated free products were obtained by Popa in \cite[Theorem 4.1]{popa93}, followed by several results of Ueda in the non-trace preserving setting \cite{ueda99, ueda03, ueda04, ueda12}.  We will use a result of Houdayer-Vaes \cite[Theorem 5.8]{houdayervaes12}, which allows for a particularly easy application in this paper.

The second aspect of free probability theory that we use is operator-valued free probability theory, as it was developed by Voiculescu \cite{voiculescu95} and Speicher \cite{speicher98}.  At the heart of this theory there lie operator-valued semicircular elements.  The von Neumann algebras generated by such elements have been described by Shlyakhtenko in \cite{shlyakhtenko99}.  We use this work in order to identify a certain free Bogoljubov crossed product as a free group factor.

Section \ref{sec:general} treats the structure of free Bogoljubov crossed products.  We obtain several different representations of free Bogoljubov crossed products associated with almost periodic orthogonal representations of $\ZZ$ in Theorem \ref{thm:afp-decomposition-ap} and Proposition \ref{prop:cocycle-crossed-product}.  We calculate the normaliser and the quasi-normaliser of the canonical abelian von Neumann subalgebra of a free Bogoljubov crossed product in Corollary \ref{cor:normaliser-quasi-normaliser} and address the question of factoriality of free Bogoljubov crossed products in Corollary \ref{cor:factorial}.  Most of the results in this section are probably folklore.

In Section \ref{sec:flexibility-ap}, we obtain isomorphism results for free Bogoljubov crossed products associated with almost periodic orthogonal representations. In particular, we classify free Bogoljubov crossed products associated with non-faithful orthogonal representations of $\ZZ$ in terms of the dimension of the representation and the index of its kernel.  They are tensor products of a diffuse abelian von Neumann algebra with an interpolated free group factor.

\begin{thmstar}[See Theorem \ref{thm:classification-periodic}]
\label{thm:classification-non-faithful-introduction}
Let $(\pi,H)$ be a non-faithful orthogonal representation of $\ZZ$ of dimension at least $2$.  Let $r = 1 + (\dim \pi - 1)/[\ZZ : \ker \pi]$.  Then
\[\Gamma(H, \ZZ, \pi)'' \cong \Linfty([0,1]) \vnt \rL \freegrp{r} \eqcomma\]
by an isomorphism carrying the subalgebra $\rL \ZZ$ of $\Gamma(H, \ZZ, \pi)''$ onto $\Linfty([0,1]) \ot \CC^{[\ZZ : \ker \pi]}$.
\end{thmstar}

For general almost periodic orthogonal representations of $\ZZ$ we can prove that the isomorphism class of the free Bogoljubov crossed product depends at most on their dimension and on the concrete subgroup of $\rS^1$ generated by the eigenvalues of their complexification.  More generally, we have the following result.
\begin{thmstar}[See Theorem \ref{thm:dependence-on-subgroup-only}]
The isomorphism class of the free Bogoljubov crossed product associated with an orthogonal representation $\pi$ of $\ZZ$ with almost periodic part $\pi_\mathrm{ap}$ depends at most on the weakly mixing part of $\pi$, the dimension of $\pi_\mathrm{ap}$ and the concrete embedding into $\rS^1$ of the group generated by the eigenvalues of the complexification of $\pi_\mathrm{ap}$.
\end{thmstar}

In contrast to the preceding result, we show later that representations with almost periodic parts of different dimension can be non-isomorphic.
\begin{thmstar}[See Theorem \ref{thm:flexibility-for-left-regular} and Theorem \ref{thm:rigidity-result}]
If $\lambda$ denotes the left regular orthogonal representation of $\ZZ$ and $\mathbb{1}$ denotes its trivial representation, then
  \[\Gamma(\ltwo(\ZZ) \oplus \CC, \ZZ, \lambda \oplus \mathbb{1})'' \cong \Gamma(\ltwo(\ZZ), \ZZ, \lambda)'' \not \cong \Gamma(\ltwo(\ZZ) \oplus \CC^2, \ZZ, \lambda \oplus 2 \cdot \mathbb{1})'' \eqstop \]
\end{thmstar}

The next results shows, however, that there are representations whose complexifications generate isomorphic, but different subgroups of $\rS^1$ and their free Bogoljubov crossed products are isomorphic nevertheless.
\begin{thmstar}[See Corollary \ref{cor:faithful-two-dimensional}]
\label{thm:faithful-two-dimensional-introduction}
  All faithful two dimensional representations of $\ZZ$ give rise to isomorphic free Bogoljubov crossed products.
\end{thmstar}

Inspired by the connection between free Bogoljubov crossed products and cores of Araki-Woods factors, and classification results for free Araki-Woods factors \cite{shlyakhtenko97}, Shlyakhtenko asked at the 2011 conference on von Neumann algebras and ergodic theory at IHP, Paris, whether for an orthogonal representation $(\pi_\RR, H_\RR)$ of $\ZZ$ the isomorphism class of $\Gamma(H_\RR, \ZZ, \pi_\RR)''$ is completely determined by the representation $\bigoplus_{n \geq 1} \pi_\RR^{\otimes n}$ up to amplification.  The present paper shows that this is not the case.  We discuss other possibilities of how a classification of free Bogoljubov crossed products could look like and put forward the following conjecture in the almost periodic case.
\begin{conjstar}[See Conjecture \ref{conj:invariant-for-ap}]
The abstract isomorphism class of the subgroup generated by the eigenvalues of the complexification of an infinite dimensional, faithful, almost periodic orthogonal representation of $\ZZ$ is a complete invariant for isomorphism of the associated free Bogoljubov crossed product.
\end{conjstar}

In Section \ref{sec:solidity}, we describe strong solidity and solidity of a free Bogoljubov crossed product $\Gamma(H, \ZZ, \pi)''$ in terms of properties of $\pi$.  The main result of \cite{houdayershlyakhtenko11} on strong solidity of free Bogoljubov crossed products is combined with ideas of Ioana \cite{ioana12} in order to obtain a bigger class of strongly solid free Bogoljubov crossed products of $\ZZ$.
\begin{thmstar}[See Theorem \ref{thm:free-bogoliubov-strong-solidity}]
  Let $(\pi, H)$ be the direct sum of a mixing representation and a representation of dimension at most one. Then $\Gamma(H, \ZZ, \pi)''$ is strongly solid.
\end{thmstar}

Orthogonal representations that have an invariant subspace of dimension two give rise to free Bogoljubov crossed products, which are obviously not strongly solid.  In particular, all almost periodic orthogonal representations are part of this class of representations.  The next theorem describes a more general class of representations of $\ZZ$ that give rise to non-solid free Bogoljubov crossed products.  If $(\pi, H)$ is a representation of $\ZZ$, we say that a non-zero subspace $K \leq H$ is rigid if there is a sequence $(n_k)_k$ in $\ZZ$ such that $\pi(n_k)|_K$ converges to $\id_K$ strongly as $n_k \ra \infty$.
\begin{thmstar}[See Theorem \ref{thm:non-solidity}]
  If the orthogonal representation $(\pi, H)$ of $\ZZ$ has a rigid subspace of dimension two, then the free Bogoljubov crossed product $\Gamma(H, \ZZ, \pi)''$ is not solid.
\end{thmstar}

We make the conjecture that this theorem describes all non-solid free Bogoljubov crossed products of the integers.
\begin{conjstar}[See Conjecture \ref{conj:solidity}]
  If $(\pi, H)$ is an orthogonal representation of $\ZZ$, then the following are equivalent.
  \begin{itemize}
  \item $\Gamma(H, \ZZ, \pi)''$ is solid.
  \item $\Gamma(H, \ZZ, \pi)''$ is strongly solid.
  \item $\pi$ has no rigid subspace of dimension two.
  \end{itemize}
\end{conjstar}

In Section \ref{sec:rigidity}, we prove a rigidity result for free Bogoljubov crossed products associated with orthogonal representations having at least a two dimensional almost periodic part. Due to the lack of invariants for bimodules over abelian von Neumann algebras, we can obtain only some non-isomorphism results.
\begin{thmstar}[See Theorem \ref{thm:rigidity-result}]
\label{thm:rigidity-result-introduction}
  No free Bogoljubov crossed product associated with a representation in the following classes is isomorphic to a free Bogoljubov crossed product associated with a representation in the other classes.
  \begin{itemize}
  \item The class of representations $\lambda \oplus \pi$, where $\lambda$ is the left regular representation of $\ZZ$ and $\pi$ is a faithful almost periodic representation of dimension at least $2$.
  \item The class of representations $\lambda \oplus \pi$, where $\lambda$ is the left regular representation of $\ZZ$ and $\pi$ is a non-faithful almost periodic representation of dimension at least $2$.
  \item The class of representations $\rho \oplus \pi$, where $\rho$ is a representation of $\ZZ$ whose spectral measure $\mu$ and all of its convolutions $\mu^{*n}$ are non-atomic and singular with respect to the Lebesgue measure on $\rS^1$ and $\pi$ is a faithful almost periodic representation of dimension at least $2$.
  \item The class of representations $\rho \oplus \pi$,  where $\rho$ is a representation of $\ZZ$ whose spectral measure $\mu$ and all of its convolutions $\mu^{*n}$ are non-atomic and singular with respect to the Lebesgue measure and $\pi$ is a non-faithful almost periodic representation of dimension at least $2$.
  \item Faithful almost periodic representations of dimension at least $2$.
  \item Non-faithful almost periodic representations of dimension at least $2$.
  \item The class of representations $\rho \oplus \pi$, where $\rho$ is mixing and $\dim \pi \leq 1$.
  \end{itemize}
\end{thmstar}

\section*{Acknowledgements}
\label{sec:acknowledgements}

We would like to thank Stefaan Vaes, our advisor, for suggesting working on this topic.  Moreover, we want to thank him for useful discussions and for helping us to improve the exposition in this article.


\section{Preliminaries}
\label{sec:preliminaries}

\subsection{Orthogonal representations of $\ZZ$ and free Bogoljubov shifts}
\label{sec:preliminaries:representations}

With a real Hilbert space $H$,  Voiculescu's \emph{free Gaussian functor} associates a von Neumann algebra $\Gamma(H)'' \cong \rL \freegrp{\dim H}$ \cite{voiculescudykemanica92}.  For every vector $\xi \in H$, we have a self-adjoint element $s(\xi) \in \Gamma(H)''$ and $\Gamma(H)''$ is generated by these elements.  If $\xi, \eta \in H$ are orthogonal then $s(\xi) + i s(\eta)$ is an element with circular distribution with respect to the trace on $\Gamma(H)''$.  In particular, the polar decomposition of $s(\xi) + i s(\eta)$ equals $a \cdot u$, where $a, u$ are $*$-free from each other, $a$ has a quarter-circular distribution and $u$ is a Haar unitary.  The free Gaussian construction $\Gamma(H)''$ acts by construction on the full Fock space $\CC \Omega \oplus \bigoplus_{n \geq 1} H^{\ot n}$ where $\Omega$ is called the vacuum vector.  It is cyclic and separating for $\Gamma(H)''$ and $\Gamma(H)''\Omega \supset H^{\ot_\mathrm{alg} n}$ for all $n \in \NN$.  Hence, for $\xi_1 \ot \dotsm \ot \xi_n \in H^{\ot_\mathrm{alg} n}$, there is a unique element $W(\xi_1 \ot \dotsm \ot \xi_n) \in \Gamma(H)''$ such that $W(\xi_1 \ot \dotsm \ot \xi_n) \Omega = \xi_1 \ot \dotsm \ot \xi_n$.

The free Gaussian construction is functorial for isometries, so that an orthogonal representation $(\pi, H)$ of a group $G$ yields a trace preserving action $G \grpaction{} \Gamma(H)''$, which is completely determined by $g \cdot s(\xi) = s(\pi(g)\xi)$.  If $\xi_1 \ot \dotsm \ot \xi_n \in H^{\ot_{\mathrm{alg}} n}$ and $g \in G$, then $g \cdot W(\xi_1 \ot \dotsm \ot \xi_n) = W(\pi(g)\xi_1 \ot \dotsm \ot \pi(g)\xi_n)$.

 An action obtained by the free Gaussian functor is called \emph{free Bogoljubov action}.  If $G \grpaction{} \Gamma(H)''$ is the free Bogoljubov action associated with $(\pi, H)$, then the representation of $G$ on $\Ltwo(\Gamma(H)'') \ominus \CC \cdot 1$ is isomorphic with $\bigoplus_{n \geq 1} \pi^{\ot n}$.  The associated von Neumann algebraic crossed product $\Gamma(H)'' \rtimes G$ of a free Bogoljubov action is denoted by $\Gamma(H, G, \pi)''$.  If there is no confusion possible, we denote $\Gamma(H, G, \pi)''$ by $M_\pi$ and the algebra $\rL G \subset \Gamma(H, G, \pi)''$ by $A_\pi$.

An orthogonal representation $(\pi, H)$ is called \emph{almost periodic} if it is the direct sum of finite dimensional representations.  It is called \emph{periodic} if the map $\pi$ has a kernel of finite index in $G$.  We call $\pi$ \emph{weakly mixing}, if it has no finite dimensional subrepresentation.  Every orthogonal representation $(\pi, H)$ is the direct sum of an almost periodic representation $(\pi_{\mathrm{ap}}, H_{\mathrm{ap}})$ and a weakly mixing representation $(\pi_{\mathrm{wm}}, H_{\mathrm{wm}})$.

Spectral theory says that unitary representations $\pi$ of $\ZZ$ correspond to pairs $(\mu, N)$, where $\mu$ is a Borel measure on $\rS^1$ and $N$ is a function with values in $\NN \cup \{\infty\}$ called the multiplicity function of $\pi$.  The measure $\mu$ and the equivalence class of $N$ up to changing it on $\mu$-negligible sets are uniquely determined by $\pi$.  Given any orthogonal representation $(\pi, H)$ of $\ZZ$, denote by $(\pi_\CC, H_\CC)$ its complexification.  Note that a pair $(\mu, N)$ as above is associated with a complexification of an orthogonal representation if and only if $\mu$ and $N$ are invariant under complex conjugation on $\rS^1 \subset \CC$.  An orthogonal representation $(\pi, H)$ is weakly mixing if and only if $\mu$ has no atoms.  It is almost periodic if and only if the measure associated with $(\pi_\CC,H_\CC)$ is completely atomic.  In this case the atoms of $\mu$ and the function $N$ together form the \emph{multiset of eigenvalues with multiplicity} of $\pi_\CC$.  Up to isomorphism, an almost periodic representation $\pi$ is uniquely determined by this multiset.

\subsection{Rigid subspaces of group representations}
\label{sec:rigid-subspaces}

A rigid subspace of an orthogonal representation $(\pi, H)$ of a discrete group $G$ is a non-zero Hilbert subspace $K \leq H$ such that there is a sequence $(g_n)_n$ of elements in $G$ tending to infinity that satisfies $\pi(g_n)\xi \lra \xi$ as $n \ra \infty$ for all $\xi \in K$.  Note that this terminology is borrowed from ergodic theory and has nothing to do with property (T).

A representation $\pi$ without any rigid subspace is called \emph{mildly mixing}.  The main source of mildly mixing representations of groups are mildly mixing actions \cite{schmidtwalters82}.  A probability measure preserving action $G \grpaction{} (X,\mu)$ has a rigid factor if there is a Borel subset $B \subset X$, $0 < \mu(B) < 1$ such that $\liminf_{g \ra \infty} \mu(B \Delta gB) = 0$.  We say that $G \grpaction{} (X, \mu)$ is mildly mixing if it has no rigid factor.

\begin{proposition}
  Let $G \grpaction{} (X,\mu)$ be a probability measure preserving action of a group $G$.  Then the Koopman representation $G \grpaction{} \Ltwo_0(X,\mu)$ is mildly mixing if and only if $G \grpaction{} (X, \mu)$ is mildly mixing.
\end{proposition}
\begin{proof}
  First assume that the Koopman representation is mildly mixing and take $B \subset X$ a Borel subset such that there is a sequence $(g_n)_n$ in $G$ going to infinity that satisfies $\mu(B \Delta g_n B) \ra 0$.  Consider the function $\xi = \mu(B) \cdot 1 - 1_B  \in \Ltwo_0(X,\mu)$.  Then
\[
  \| \xi - g_n \xi \|_2^2
  =
  \| 1_{g_nB} - 1_B \|_2^2
  =
  \mu(B \Delta g_n B)
  \ra
  0
  \eqstop
\]
By mild mixing of $G \grpaction{} \Ltwo_0(X, \mu)$, it follows that $\xi = 0$, so $\mu(B) \in \{0,1\}$.  Hence $G \grpaction{} (X,\mu)$ is mildly mixing.

For the converse implication assume that there is a sequence $(g_n)_n$ in $G$ tending to infinity such that there is a unit vector $\xi \in \Ltwo_0(X,\mu)$ that satisfies $g_n \xi \ra \xi$.  We have to show that $G \grpaction{} (X,\mu)$ has a rigid factor.  Replacing $\xi$ by its real part, we may assume that it takes only real values.  For $\delta > 0$ define $A_\delta = \{x \amid \xi(x) \geq \delta\}$ and $B_\delta = \{x \amid \xi(x) > \delta\}$.  Since $\int_X \xi(x) \rmd \mu(x) = 0$, there is some $\delta > 0$ such that $0 < \mu(A_\delta) < 1$.  

Take $\varepsilon > 0$.  We have $\bigcap_{\delta' < \delta} B_{\delta'} = A_\delta$, so that we can choose $\delta' < \delta$ such that $\mu(B_{\delta'} \setminus A_\delta) < \varepsilon/4$.  Take $N \in \NN$ such that for all $n \geq N$ we have $\|\xi - g_n \xi\| < (\delta - \delta') \cdot \varepsilon / 4$.  Then for all $n \geq N$, we have
\begin{align*}
  \mu(A_\delta \Delta g_n A_\delta)
  & =
  \mu(A_\delta \setminus g_n A_\delta) +   \mu(A_\delta \setminus g_n^{-1} A_\delta) \\
  & <
  \mu(A_\delta \setminus g_n B_{\delta'}) +   \mu(A_\delta \setminus g_n^{-1} B_{\delta'}) + \frac{\varepsilon}{2} \\
  & \leq
  \frac{1}{(\delta - \delta')^2}\left ( \,
    \int_{A_\delta \setminus g_n B_{\delta'}} | \xi(x) - g_n \xi(x)|^2 \rmd x
    +
    \int_{A_\delta \setminus g_n^{-1}B_{\delta'}} |\xi(x) - g_n^{-1}\xi(x)|^2 \rmd x
  \right )
  + \frac{\varepsilon}{2} \\
  & \leq
  \frac{2}{(\delta - \delta')^2} \int_X |\xi(x) - g_n \xi(x)|^2 \rmd \mu(x) + \frac{\varepsilon}{2}\\
  & <
  \varepsilon
  \eqstop
\end{align*}
It follows that $\mu(A_\delta \Delta g_n A_\delta) \ra 0$ as $n \ra \infty$.  So $G \grpaction{} (X,\mu)$ is not mildly mixing.
\end{proof}

\subsection{Bimodules over von Neumann algebras}
\label{sec:bimodules}

Let $M$, $N$ be von Neumann algebras.  An $M$-$N$-bimodule is a Hilbert space $\cH$ with a normal $*$-representation of $\lambda: M \ra \bo(\cH)$ and a normal anti-$*$-representation $\rho:N \ra \bo(H)$ such that $\lambda(x) \rho(y) = \rho(y) \lambda(x)$ for all $x \in M$, $y \in N$.  If $M$, $N$ are tracial, then we have $\lmod{M}{\cH} \cong \lmod{M}{(\Ltwo(M) \ot \ltwo(\NN)^*)p}$ with $p \in M \ot \bo(\ltwo(\NN))$.  The  left dimension $\dim_{M -} \cH$ of $\lmod{M}{\cH}$ is $(\tau_M \ot \Tr)(p)$ by definition.  Similarly, we define the right dimension $\dim_{\, - N} \cH$ of $\rmod{\cH}{N}$.  We say that $\bim{M}{\cH}{N}$ is left finite, if it has finite left dimension, we call it right finite if it has finite right dimension and we say that $\cH$ is a finite index $M$-$N$-bimodule, if its left and right dimension are both finite.

If $A, B \subset M$ are abelian von Neumann algebras and $\bim{A}{\cH}{B} \subset \Ltwo(M)$ is a finite index bimodule, then there are non-zero projections $p \in A, q \in B$, a finite index inclusion $\phi: pA \ra qB$ and a non-zero partial isometry $v \in p M q$ such that $av = v\phi(a)$ for all $a \in pA$. Since $\phi$ is a finite index inclusion, we can cut down $p$ and $q$ so as to assume that $\phi$ is an isomorphism.

\subsection{The measure associated with a bimodule over an abelian von Neumann algebra}
\label{sec:measure-from-bimodule}

We describe bimodules over abelian von Neumann algebras, as in \cite[V. Appendix B]{connes94}.  Compare also with \cite[Section 3]{neshveyevstormer02} concerning our formulation.  Let $A \cong \Linfty(X, \mu)$ be an abelian von Neumann algebra and $\bim{A}{\cH}{A}$ an $A$-$A$-bimodule such that $\lambda, \rho: A \ra \bo(\cH)$ are faithful.  Then the two inclusions $\lambda, \rho: A \ra \bo(\cH)$ generate an abelian von Neumann algebra $\cA$.  Writing $[\nu]$ for the class of a measure $\nu$ and $p_1,p_2$ for the projections on the two factors of $X \times X$, we can identify $\cA \cong \Linfty(X \times X, \nu)$ where $[\nu]$ is subject to the condition $(p_1)_*([\nu]) = (p_2)_*([\nu]) = [\mu]$.  We can disintegrate $\cH$ with respect to $\nu$ and obtain a decomposition $\cH = \int^\oplus_{X \times X} \cH_{x_1,x_2} \rmd \nu(x_1,x_2)$.  Let $N: X \times X \ra \NN \cup \{\infty\}$ be the dimension function $\cH_{x_1,x_2} \mapsto \dim_\CC \cH_{x_1,x_2}$.  Then $N$ is unique up to changing it on $\nu$-negligible sets and the triple $(X, [\nu], N)$ is a conjugacy invariant for $\bim{A}{\cH}{A}$ in the following sense.  Let $(X, [\nu_X], N_X)$ and $(Y, [\nu_Y], N_Y)$ be triples as before associated with bimodules $\cH_X$ and $\cH_Y$ over $A = \Linfty(X,\mu_x)$ and $B = \Linfty(Y, \mu_y)$, respectively.  A measurable isomorphisms  $\Delta: (X, [\mu_X]) \ra (Y, [\mu_Y])$  such that $(\Delta \times \Delta)_*([\nu_X]) = [\nu_Y]$  and $N_Y \circ (\Delta \times \Delta) = N_X$ $\nu_Y$-almost everywhere induces an isomorphism $\theta:A \ra B$ and a unitary isomorphism $U: \cH_X \ra \cH_Y$ satisfying 
\[U\lambda_X(a) = \lambda_Y(\theta(a))U \text{ and } U\rho_X(a) = \rho_Y(\theta(a))U \text{ for all } a \in A \eqstop\]
Moreover, any such pair $(U, \theta)$ arises this way.  The proof of this fact works similar to that of \cite{neshveyevstormer02}.

Let $\bim{A}{\cH}{A}$ be an $A$-$A$-bimodule and identify $A \cong \Linfty(X, \mu)$.  Denote by $(X, [\nu], N)$ the spectral invariant of $\bim{A}{\cH}{A}$ as described in the previous paragraph.  If $p = 1_Y \in A$ is a non-zero projection, then it follows right away that the spectral invariant associated with $\bim{pA}{(p\cH p)}{pA}$ equals $(Y, [\nu|_{Y \times Y}], N|_{Y \times Y})$.

Let $\ZZ \grpaction{} P$ be an action of $\ZZ$ on a tracial von Neumann algebra $P$ and $M = P \rtimes \ZZ$.  Let $(\mu, N_{\pi})$ denote the spectral invariant of the representation $\pi$ on $\Ltwo(P) \ominus \CC 1$ associated with the action of $\ZZ$ on $P$.  Write $A = \rL \ZZ \cong \Linfty(\rS^1)$, where the identification is given by the Fourier transform.  We describe the spectral invariant $(\rS^1, [\nu], N)$ of the $A$-$A$-bimodule $\Ltwo(M) \ominus \Ltwo(A)$ in terms of $(\pi, N_{\pi})$.  

We first calculate the measure $\nu_{\xi \ot \delta_n}$ on $\rS^1 \times \rS^1$ defined by
\[ \int_{\rS^1 \times \rS^1}s^at^b \, \rmd \nu_{\xi \ot \delta_n}(s,t) = \langle u_a (\xi \otimes \delta_n) u_b, \xi \otimes \delta_n \rangle \eqcomma\]
with $a,b \in \ZZ$, $\xi \in \Ltwo(P) \ominus \CC 1$ and $\delta_n \in \ltwo(\ZZ)$ the canonical basis element associated with $n \in \ZZ$.  Denote by $\mu_\xi$ the measure on $\rS^1$ defined by
\[ \int_{\rS^1} s^a \, \rmd \mu_\xi(s) = \langle \pi(a) \xi , \xi \rangle \eqstop\]
We obtain for $a,b \in \ZZ$, $\xi \in \Ltwo(P) \ominus \CC 1$ and $n \in \ZZ$
\begin{align*}
  \int_{\rS^1 \times \rS^1} s^at^b \, \rmd \nu_{\xi \ot \delta_n}(s,t)
  & = \langle u_a (\xi \otimes \delta_n) u_b, \xi \otimes \delta_n \rangle \\
  & = \delta_{a, -b} \langle \pi(a) \xi, \xi \rangle \\
  & = \delta_{a, -b} \int_{\rS^1} s^a \, \rmd \mu_\xi(s) \\
  & = \int_{\rS^1 \times \rS^1} s^at^{a + b} \, \rmd(\mu_\xi \ot \lambda)(s,t) \\
  & = \int_{\rS^1 \times \rS^1} s^at^b \, \rmd T_*(\mu_\xi \ot \lambda)(s,t) \eqcomma
\end{align*}
where $T: \rS^1 \times \rS^1 \ra \rS^1 \times \rS^1: (s,t) \mapsto (s, st)$.  So $\nu_{\xi \ot \delta_n} = T_*(\mu_\xi \ot \lambda)$ for all $\xi \in \Ltwo(M) \ominus \CC 1$ and for all $n \in \ZZ$.  It follows that $[\nu] = T_*([\mu \ot \lambda])$. 

We calculate the multiplicity function $N$ of $\Ltwo(M) \ominus \Ltwo(A)$ in terms of $N_{\pi}$.  Let $Y_n, n \in \NN \cup \{\infty\}$ be pairwise disjoint Borel subsets of $\rS^1$ such that $N_{\pi}|_{Y_n} = n$ for all $n$.  There is a basis $(\xi_{n,k})_{0 \leq k < n \in \NN \cup \{\infty\}}$ of $\Ltwo(P) \ominus \CC$ such that $\mu_{\xi_{n,k}}$ is has support equal to $Y_n$.  So $\xi_{n,k} \ot \delta_l$ with $l \in \ZZ$ and $0 \leq k < n \in \NN \cup \{\infty\}$ is a basis of $\Ltwo(M) \ominus \Ltwo(A)$.  Write $Z_n = T(Y_n \times \rS^1)$.  Then
\[\int_{Z_n} s^at^b \, \rmd \nu_{\xi_{n,k} \ot \delta_l} (s,t)
   =
  \int_{Y_n \times \rS^1} s^at^{a + b} \, \rmd (\mu_{\xi_{n,k}} \ot \lambda)(s,t)
  \eqcomma \]
so the support of $\nu_{\xi_{n,k} \ot \delta_l}$ is equal to $Z_n$.  As a consequence, $N|_{Z_n} = n$ for all $n \in \NN \cup \{\infty\}$.  We obtain the following proposition.

\begin{proposition}
\label{prop:spectral-decomposition-FBCP}
  Let $(\mu, N)$ be a symmetric measure with multiplicity function on $\rS^1$ having at least one atom and let $\pi$ be the orthogonal representation of $\ZZ$ on $H = \Ltwo_\RR(\rS^1, \mu, N)$ given by $\pi(1) f = \id_{\rS^1} \cdot f$.  Identifying $\rL \ZZ \cong \Linfty(\rS^1)$ via the Fourier transform, the multiplicity function of the bimodule $\bim{\Linfty(\rS^1)}{\Gamma(H, \ZZ, \pi)''}{\Linfty(\rS^1)}$ is equal to $\infty$ almost everywhere.
\end{proposition}
\begin{proof}
  We have $\Gamma(H, \ZZ, \pi)'' = \Gamma(H)'' \rtimes \ZZ$, where the crossed product is taken with respect to the free Bogoljubov action of $\ZZ$ on $\Gamma(H)''$, which has $\oplus_{n \geq 1}\pi^{\ot n}$ as its associated representation on $\Ltwo(\Gamma(H)'') \ominus \CC \cdot 1$.  If $a$ is an atom of $\mu$, then also $\ol{a}$ is one.  Denote by $\chi_a$ the character of $\ZZ$ defined by $\widehat{\ZZ} \cong \rS^1$.  We have $\pi = \pi \ot (\chi_a)^n \ot (\chi_{\ol{a}})^n \leq \pi^{\ot 2n}$.  As a consequence, the multiplicity function of $\oplus_{n \geq 1} \pi^{\ot n}$ is equal to $\infty$ almost everywhere.  So, by the calculations preceding the remark, this is also the case for the multiplicity function of the bimodule $\bim{\Linfty(\rS^1)}{\Ltwo(\Gamma(H, \ZZ, \pi)'')}{\Linfty(\rS^1)}$.
\end{proof}

\begin{proposition}
\label{prop:disintegration-spectral-measure-FBCP}
  The disintegration of $[\nu]$ with respect to the projection onto the first component of $\rS^1 \times \rS^1$ is given by $[\nu] = \int [\mu * \delta_s] \, \rmd \lambda (s)$.
\end{proposition}
\begin{proof}
Let $Y, Z \subset \rS^1$ be Borel subsets and denote by $(\mu_s)_{s \in \rS^1}$ the constant field of measures with value $\mu$.
\begin{align*}
  (T_* \left (\int_{\rS^1} \mu_s \, \rmd \lambda(s) \right ))( Y \times Z)
  & =
  \int_Y \mu(Z \cdot s^{-1}) \, \rmd \lambda(s) \\
  & =
  \int_Y \mu * \delta_s(Z) \, \rmd \lambda(s) \\
  & =
  (\int_{\rS^1} \mu * \delta_s \, \rmd \lambda(s))(Y \times Z) \eqstop
\end{align*}
This finishes the proof.
\end{proof}

\subsection{Amalgamated free products over finite dimensional algebras}
\label{sec:afp-over-finite-dimensional-algebras}

Let $\cR_2$ denote the class of finite direct sums of hyperfinite von Neumann algebras and interpolated free group factors, equipped with a normal, faithful tracial state.  In \cite[Theorem 4.5]{dykemaredelmeier11}, amalgamated free products of elements of $\cR_2$ over finite dimensional tracial von Neumann subalgebras were shown to be in $\cR_2$ again.  Moreover, their \emph{free dimension} in the sense of Dykema \cite{dykema11-finite-vnalg} was calculated in terms of the free dimension of the factors and of the amalgam of the amalgamated free product.  We explain the free dimension and Theorem 4.5 of \cite{dykemaredelmeier11}.

The free dimension of a set of generators of a von Neumann algebra $M \in \cR_2$ is used to keep track of the parameter of interpolated free group factors.  If an interpolated free group factor has a generating sets of free dimension $r$, then it is isomorphic to $\rL \freegrp{r}$.  Following \cite{dykemaredelmeier11}, we define the class $\cF_d \subset \cR_2$, $d \in \RR_{>0}$ as the class of von Neumann algebras
\[M = D \oplus \bigoplus_{i \in I} p_i\rL \freegrp{r_i} \oplus \bigoplus_{j \in J} q_j\Cmat{n_j} \eqcomma\]
where 
\begin{itemize}
\item $p_i$ is the unit of $\rL \freegrp{r_i}$ and $q_j$ the unit of $\Cmat{n_j}$,
\item $t_i = \tau_M(p_i)$, $s_j = \frac{\tau_M(q_j)}{n_j}$ and $D$ is a diffuse hyperfinite von Neumann algebra and
\item $1 + \sum_it_i^2(r_i - 1) - \sum_j s_j^2 = d$.
\end{itemize}

Theorem 4.5 of \cite{dykemaredelmeier11} says that if $M = M_1 *_A M_2$ with $M_1, M_2 \in \cR_2$ and $A$ a finite dimensional tracial von Neumann algebra, then $M \in \cR_2$.  Moreover, if $M_1 \in \cF_{d_1}$, $M_2 \in \cF_{d_2}$ and $A \in \cF_d$, then $M \in \cF_{d_1 + d_2 - d}$.  We will use the following special case.
\begin{theorem}[See Theorem 4.5 of \cite{dykemaredelmeier11}]
\label{thm:dykemaredelmeier}
  Let $M_1 \in \cF_{d_1}$ and $M_2 \in \cF_{d_2}$ and $A \in \cF_{d}$ a common finite dimensional subalgebra of $M_1$ and $M_2$. If $M = M_1 *_A M_2$ is a non-amenable factor, then $M \cong \rL \freegrp{r}$ with $r = d_1 + d_2 - d$. 
\end{theorem}

 We will use this result in combination with a special case Theorem 5.8 of \cite{houdayervaes12}.

 \begin{theorem}[See Theorem 5.8 of \cite{houdayervaes12}]
\label{thm:houdayervaes}
   Let $M_1$, $M_2$ be diffuse von Neumann algebras and $A$ a common finite dimensional subalgebra.  If $\cZ(M_1) \cap \cZ(M_2) \cap \cZ(A) = \CC 1$, then $M_1 *_A M_2$ is a non-amenable factor.
 \end{theorem}

\subsection{Operator valued semicircular random variables}
\label{sec:operator-valued-semicirculars}

Given an inclusion of von Neumann algebras $A \subset M$ with conditional expectation $\rE:M \ra A$, we say that an element $X$ of $M$ is a random variable with $A$-valued distribution
\[
  \phi^{(n)}_{(X,A)}:
  A \times \dotsm \times A \ra A:
  (a_1, \dotsc, a_n) \mapsto \rE(X a_1 X \dotsm a_n X)
  \eqcomma \quad
  n \in \NN
  \eqstop
\]
If $\mathrm{NC}(n)$ denotes the set of all non-crossing partitions on $n$ points, then we can use the framework of \emph{operator-valued multiplicative function} of Speicher \cite[Chapter II]{speicher98} in order to write the operator-valued free cumulants of $X$ as the unique maps $c^{(n)}: A \times \dotsm \times A \ra A$ satisfying
\[
  \phi_{(X,A)}^{(n)}(a_1, \dotsc, a_n)
  =
  \sum_{\pi \in \mathrm{NC}}{
    c^{(\pi)}_{(X,A)}(a_1, \dotsc, a_n)
  }
  \eqcomma
\]
where $c^{(\pi)}_{(X,A)}$ is defined recursively over the block structure of $\pi$.  If $b = \{i, i + 1, \dotsc , i +k\}$ is an interval of $\pi \in NC(n+1)$, then we put
\[
  c^{(\pi)}_{(X,A)}(a_1, \dotsc, a_n)
  = 
  c^{(\pi \setminus b)}_{(X,A)}
    (a_1, \dotsc,
    a_{i-1} c^{(k)}_{(X,A)}(a_i, \dotsc, a_{i + k - 1}) a_{i + k},
    \dotsc , a_n)
  \eqstop
\]

If $\eta:A \ra A$ is a completely positive map, then an $A$-valued random variable $X \in M$ is called $A$-valued semicircular with distribution $\eta$, if $c^{(1)}_{(X,A)}(a) = \eta(a)$ and $c^{(n)}_{(X,A)} = 0$ for all $n \neq 1$.  We will need the following proposition.
\begin{proposition}[See Example 3.3(a) in \cite{shlyakhtenko99}]
  If $A \cong \rL \ZZ$ and $X$ is an $A$-valued semicircular with distribution $\eta = \tau: A \ra \CC \subset A$, then $\Wstar(X, A) \cong \rL \freegrp{2}$ where $u_1 \in A$ is identified with one canonical generator of $\rL \freegrp{2}$.
\end{proposition}

\subsection{Deformation/Rigidity}
\label{sec:preliminaries:deformation-rigidty}

Let $A \subset M$ be an inclusion of von Neumann algebras.  The normaliser of $A$ in $M$, denoted by $\Norm_M(A)''$, is the von Neumann algebra generated by all unitaries $u \in M$ satisfying $uAu^* = A$.  The quasi-normaliser of $A$ in $M$ is the von Neumann algebra $\QN_M(A)''$ generated by all elements $x \in M$ such that there are $a_1, \dotsc , a_n$ and $b_1, \dotsc, b_m$ satisfying $Nx \subset \sum_i a_iN$ and $xN \subset \sum_i N b_i$.

The following notion was introduced in \cite[Theorem 2.1 and Corollary 2.3]{popa06_2}.  If $M$ is a tracial von Neumann algebra, $A, B \subset M$ are von Neumann subalgebras, we say that \emph{$A$ embeds into $B$ inside $M$} if there is a right finite $A$-$B$-subbimodule of $\Ltwo(M)$.  In this case, we write $A \prec_M B$. If every $A$-$M$-subbimodule of $\Ltwo(M)$ contains a right finite $A$-$B$-subbimodule, then we say that \emph{$A$ fully embeds into $B$ inside $M$} and write $A \prec_M^\rmf B$.

If $A, B \subset (M, \tau)$ is an inclusion of tracial von Neumann algebras, we say that $A$ is \emph{amenable relative to $B$ inside $M$}, if there is an $A$ central state $\varphi$ on the basic construction $\langle M, e_B \rangle$ such that $\varphi|_M = \tau$.  If $A$ is amenable relative to an amenable subalgebra, then it is amenable itself.

We will use the following theorem from \cite{houdayershlyakhtenko11}.  It is proven there for unital von Neumann subalgebras only, but the same proof shows, that it's true for non-unital von Neumann subalgebras.

\begin{theorem}[Theorem 3.5 of \cite{houdayershlyakhtenko11}]
\label{thm:embedding-and-normaliser}
Let $G$ be an amenable group with an orthogonal representation $(\pi, H)$ and write $M = \Gamma(H, G, \pi)''$.  Let $p \in M$ be a non-zero projection and $P \subset pMp$ a von Neumann subalgebra such that $P \not \prec_M \rL G$.  Then $\Norm_{pMp}(P)''$ is amenable.
\end{theorem}

Since we need full embedding of subalgebras in this paper, let us deduce a corollary of the previous theorem.

\begin{corollary}[See Theorem 3.5 of \cite{houdayershlyakhtenko11}]
\label{cor:full-embedding}
  Let $G$ be an amenable group with an orthogonal representation $(\pi, H)$ and write $M = \Gamma(H, G, \pi)''$.  Let $P \subset M$ be a von Neumann subalgebra such that $\Norm_M(P)''$ has no amenable direct summand.  Then $P \prec_M^\rmf \rL G$.
\end{corollary}
\begin{proof}
  Take $P \subset M$ as in the statement and let us assume for a contradiction that $P \not \prec_M^\rmf \rL G$.  Let $p \in P' \cap M$ be the maximal projection such that $pP \not \prec_M \rL G$.  Then $p \in \cZ(\Norm_M(P)'')$.  By \cite[Lemma 3.5]{popa06_2}, we have $\Norm_{pMp}(pP)'' \supset p \Norm_M(P)'' p$.  By Theorem \ref{thm:embedding-and-normaliser}, $\Norm_{pMp}(pP)''$ is amenable.  So $\Norm_M(P)''$ has an amenable direct summand.  This is contradiction.
\end{proof}

The next theorem, due to Vaes, allows us to obtain from intertwining bimodules a much better behaved finite index bimodule.

\begin{proposition}[Proposition 3.5 of \cite{vaes09}]
\label{prop:fi-bimodule}
Let $M$ be a tracial von Neumann algebra and suppose that $A,B \subset M$ are von Neumann subalgebras that satisfy the following conditions.
\begin{itemize}
\item $A \prec_M B$ and $B \prec_M^{\rmf} A$.
\item If $\cH \leq \Ltwo(M)$ is an $A$-$A$ bimodule with finite right dimension, then $\cH \leq \Ltwo(\QN_M(A)'')$.
\end{itemize}
Then there is a finite index $A$-$B$-subbimodule of $\Ltwo(M)$.
\end{proposition}

\subsubsection{Deformation/Rigidity for amalgamated free products}
\label{sec:deformation-rigidity-afps}

We will make use of the following results, which control relative commutants in amalgamated free products.
\begin{theorem}[See Theorem 1.1 of \cite{ioanapetersonpopa08}]
\label{thm:ipp}
Let $M = M_1 *_A M_2$ be an amalgamated free product of tracial von Neumann algebras and $p \in M_1$ a non-zero projection. If $Q \subset pM_1p$ is a von Neumann subalgebra such that $Q \not \prec_{M_1} A$, then $Q' \cap pMp = Q' \cap pM_1 p$.  
\end{theorem}

\begin{theorem}[See Theorem 6.3 in \cite{ioana12}]
\label{thm:ioana-control-relative-commutant}
  Let $M = M_1 *_A M_2$ be an amalgamated free product of tracial von Neumann algebras and $p \in M$.  Let $Q \subset pMp$ an arbitrary von Neumann subalgebra and $\omega$ a non-principal ultrafilter.  Denote by $B$ the von Neumann algebra generated by $A^\omega$ and $M$.  One of the following statements is true.
  \begin{itemize}
  \item $Q' \cap (pMp)^\omega \subset B$ and $Q' \cap (pMp)^\omega \prec_{M^\omega} A^\omega$,
  \item $\Norm_{pMp}(Q)'' \prec M_i$, for some $i \in \{1,2\}$ or
  \item $Qe$ is amenable relative to $A$ for some non-zero projection $e \in \cZ(Q' \cap pMp)$.
  \end{itemize}
\end{theorem}

Also, we will need one result on relative commutants in ultrapowers.
\begin{lemma}[See Lemma 2.7 in \cite{ioana12}]
\label{lem:ioana-relative-commutant-in-ultrapower}
  Let $M$ be a tracial von Neumann algebra, $p \in M$ a non-zero projection, $P \subset pMp$ and $\omega$ a non-principal ultrafilter.  There is a decomposition $p = e + f$, where $e, f \in \cZ(P' \cap (pMp)^\omega) \cap \cZ(P' \cap pMp)$ are projections such that
  \begin{itemize}
  \item $e(P' \cap (pMp)^\omega) = e(P' \cap pMp)$ and this algebra is completely atomic and
  \item $f(P' \cap (pMp)^\omega)$ is diffuse.
  \end{itemize}
\end{lemma}

A tracial inclusion $B \subset M$ of von Neumann algebras is called \emph{mixing} if for all sequences $(x_n)_n$ in the unit ball $(B)_1$ that go to $0$ weakly and for all $y,z \in M \ominus B$, we have
\[\|\rE_B(yx_nz)\|_2 \ra 0 \text{ if } n \ra \infty \eqstop\]
If a subalgebra is mixing, we can control the normaliser of algebras embedding into it.
\begin{lemma}[See Lemma 9.4 in \cite{ioana12}]
\label{lem:ioana-control-normalizer-mixing}
  Let $B \subset M$ be a mixing inclusion of tracial von Neumann algebras.  Let $p \in M$ be a projection and $Q \subset pMp$.  If $Q \prec_M B$, then $\Norm_M(Q)'' \prec_M B$.
\end{lemma}

Finally, we will use two theorems on intertwining in amalgamated free products from the work of Ioana \cite{ioana12}.  This theorem is stated in \cite{ioana12} for unital inclusions into amalgamated free products, but it remains valid in the more general case.
\begin{theorem}[See Theorem 1.6 in \cite{ioana12}]
\label{thm:ioana-main}
Let $M = M_1 *_A M_2$ be an amalgamated free product of tracial von Neumann algebras, $p \in M$ a projection and $Q \subset pMp$ an amenable von Neumann subalgebra.  Denote by $P = \Norm_{pMp}(Q)''$ the normaliser of $Q$ inside $pMp$ and assume that $P' \cap (pMp)^\omega = \CC p$ for some non-principal ultrafilter $\omega$.  Then, one of the following holds.
\begin{itemize}
\item $Q \prec A$,
\item $P \prec M_i$, for some $i \in \{1,2\}$ or
\item $P$ is amenable relative to $A$.
\end{itemize}
\end{theorem}

\begin{theorem}[See Theorem 9.5 in \cite{ioana12}]
\label{thm:approximate-commutators-and-mixing-subaglebras}
Let $B \subset M$ be a mixing inclusion of von Neumann algebras.  Take a non-principal ultrafilter $\omega$, a projection $p \in M$ and let $P \subset pMp$ be a von Neumann subalgebra such that $P' \cap (pMp)^\omega$ is diffuse and $P' \cap (pMp)^\omega \prec_{M^\omega} B^\omega$.  Then $P \prec_M B$.  
\end{theorem}


\section{General structure of $\Gamma(H, \ZZ, \pi)''$}
\label{sec:general}

Recall that we write $M_\pi$ for $\Gamma(H, \ZZ, \pi)''$.  The decomposition of orthogonal representations into almost periodic and weakly mixing part, also gives rise to a decomposition of their free Bogoljubov crossed products.

\begin{remark}
\label{rem:decomposition-wm-ap}
Let $(\pi, H)$ be an orthogonal representation of a discrete group $G$.  Then
\[\Gamma(H)'' \cong \Gamma(H_{\mathrm{ap}})'' * \Gamma(H_{\mathrm{wm}})''\]
and so we get a decomposition
\[
  M_\pi
  =
  \Gamma(H)'' \rtimes G
  \cong
  (\Gamma(H_{\mathrm{ap}})'' \rtimes G) *_{\rL G} (\Gamma(H_{\mathrm{wm}})'' \rtimes G)
  \eqstop
\]
More generally, if $\pi = \bigoplus_i \pi_i$, then $M_\pi \cong *_{\rL G, i} M_{\pi_i}$.
\end{remark}

\subsection{$\Gamma(H, \ZZ, \pi)''$ for almost periodic representations}
\label{sec:almost-periodic}

If not mentioned explicitly, $\pi$ denotes an almost periodic orthogonal representation of $\ZZ$ in this section.  Recall that an irreducible almost periodic orthogonal representation of $\ZZ$ has dimension $1$ if and only if its eigenvalue is $1$ or $-1$.  In all other cases, it has dimension $2$ and its complexification has a pair of conjugate eigenvalues $\lambda, \ol{\lambda} \in \rS^1 \setminus \{1, -1\}$.

\begin{notation}
  We denote by $\rL \ZZ \rtimes_\lambda \ZZ$, $\lambda \in \rS^1$ the crossed product by the action of $\ZZ$ on $\rL \ZZ$ where $1 \in \ZZ$ acts by multiplying the canonical generator of $\rL \ZZ$ with $\lambda$.  This is isomorphic to the crossed products $\Linfty (\rS^1) \rtimes_\lambda \ZZ$ and $\ZZ \ltimes_\lambda \Linfty(\rS^1)$, where $\ZZ$ acts on $\rS^1$ by rotation by $\lambda$.  Moreover, $1 \ot \rL \ZZ$ is carried onto $1 \ot \rL \ZZ$ and $1 \ot \Linfty(\rS^1)$, respectively, under this isomorphism.
\end{notation}

\begin{theorem}
\label{thm:afp-decomposition-ap}
  Let $\pi$ be an almost periodic orthogonal representation of $\ZZ$.  Let $\lambda_i, \ol{\lambda_i}$, $0 \leq i < n_1 \in \NN \cup \{\infty\}$ be an enumeration of all eigenvalues in $\rS^1 \setminus \{1, -1\}$ of the complexification of $\pi$.  Denote by $n_2$ and $m_0$ the multiplicity of $-1$ and $1$, respectively, as an eigenvalues of $\pi$.  Note that $\dim \pi = 2n_1 + n_2 + m_0$ and write $n = n_1 + n_2$, $m = n_1 + m_0$.  Then
\begin{align*}
  M_\pi
  & \cong
  (\rL \freegrp{m} \vnt \rL \ZZ) *_{1 \ot \rL \ZZ} (\rL \freegrp{n} \rtimes_\alpha \ZZ) \\
  & \cong 
  (\rL \freegrp{m} \vnt \Linfty(\rS^1)) *_{1 \ot \Linfty(\rS^1)} (\freegrp{n} \ltimes_\beta \Linfty(\rS^1))
 \eqcomma
\end{align*}
where, denoting by $g_i$, $0 \leq i < n_1$, and $h_i$, $0 \leq i < n_2$, the canonical basis of $\freegrp{n_1 + n_2} \cong \freegrp{n}$
\begin{itemize}
\item $\alpha(1)$ acts on $u_{g_i}$ by multiplication with $\lambda_i$ for $0 \leq i < n_1$,
\item $\alpha(1)$ acts on $u_{h_i}$ by multiplication with $-1$ for $0 \leq i < n_2$,
\item $\beta(g_i)$ acts on $\rS^1$ by multiplication with $\lambda_i$ for $0 \leq i < n_1$,
\item $\beta(h_i)$ acts on $\rS^1$ by multiplication with $-1$ for $0 \leq i < n_2$.
\end{itemize}
 Moreover, $\Gamma(H_\pi)'' \cong \rL(\freegrp{m + n})$ under this identification and $A_\pi$ is carried onto $\rL \ZZ$ and $\Linfty(\rS^1)$, respectively.

\end{theorem}
\begin{proof}
If $\pi$ is the trivial representation, then $M_\pi \cong \rL \freegrp{\dim \pi} \vnt \rL \ZZ$.  If $\pi$ is the one dimensional representation with eigenvalue $-1$, then 
\[(A_\pi \subset M_\pi) \cong (1 \ot \rL \ZZ \subset \rL \ZZ \rtimes_{-1} \ZZ) \eqstop\]
Let $\pi$ be an irreducible two dimensional representation of $\ZZ$ with eigenvalues $\lambda, \ol{\lambda} \in \rS^1$ of its complexification.  We show that 
\[M_\pi \cong (\rL \ZZ \vnt \rL \ZZ) *_{1 \ot \rL \ZZ} (\rL \ZZ \rtimes_\lambda \ZZ)\]
where the inclusion $1 \ot \rL \ZZ \subset (\rL \ZZ \vnt \rL \ZZ) *_{1 \ot \rL \ZZ} (\rL \ZZ \rtimes_\lambda \ZZ)$ is identified with $A_\pi \subset M_\pi$ under this isomorphism. Indeed, let $\xi, \eta \in H$ be orthogonal such that $\xi + i\eta$ is an eigenvector with eigenvalue $\lambda$ for the complexification of $\pi$.  Write $c = s(\xi) + is(\eta)$.  Then $c$ is a circular element in $M_\pi$ such that $\alpha_\pi(1)c = \lambda c$. Let $c = ua$ be the polar decomposition.  As explained in Section \ref{sec:preliminaries:representations}, $u$ is a Haar unitary and $a$ has quarter-circular distribution and they are $*$-free from each other.  Moreover, $\alpha_\pi(1)a = a$ and thus $\alpha_\pi(1)u = \lambda u$, by uniqueness of the polar decomposition.  So the von Neumann algebra generated by $a$, $u$ and $\rL \ZZ$ is isomorphic to $(\rL \ZZ \vnt \rL \ZZ) *_{1 \ot \rL \ZZ} (\rL \ZZ \rtimes \rL \ZZ)$ and $A_\pi$ is identified with the subalgebra $1 \ot \rL \ZZ$.  This gives the first isomorphism in the statement of the theorem.  Since $\rL \ZZ \rtimes_\lambda \ZZ \cong \ZZ \ltimes_\lambda \Linfty(\rS^1)$ sending $1 \ot \rL \ZZ$ onto $1 \ot \Linfty(\rS^1)$ via the Fourier transform, we also obtain the second isomorphism in the statement of the theorem.

The case of a general almost periodic orthogonal representation $\pi$ follows by considering its decomposition into irreducible components as in Remark \ref{rem:decomposition-wm-ap}.  Indeed, denote by
\[
  \pi
  =
  \bigoplus_{0 \leq i < n_1} \pi_{i, c} \oplus
  \bigoplus_{0 \leq i < n_2} \pi_{i, -1} \oplus
  \bigoplus_{0 \leq i < m_0} \pi_{i,1}
\]
the decomposition of $\pi$ into irreducible components.  Here $\pi_{i,c}$ has dimension $2$ with eigenvalues $\lambda_i, \ol{\lambda_i}$ of $(\pi_{i, c})_\CC$ and $\pi_{i,-1}$ has eigenvalue $-1$ and $\pi_{i,1}$ is the trivial representation. Then
\begin{align*}
  M_\pi
  & \cong
  (\Asterisk_{0 \leq i < n_1} M_{\pi_{i,c}}) *_{A_\pi}   (\Asterisk_{0 \leq i < n_2, A} M_{\pi_{i,-1}}) *_{A_\pi}   (\Asterisk_{0 \leq i < m_0, A} M_{\pi_{i,c}}) \\
  & \cong
  \left ( \Asterisk_{0 \leq i < n_1, 1 \ot \Linfty(\rS^1)} (\rL \ZZ \ot \Linfty(\rS^1)) *_{1 \ot \Linfty(\rS^1)} (\ZZ \ltimes_{\lambda_i} \Linfty(\rS^1)) \right ) \\
    & \qquad *_{1 \ot \Linfty(\rS^1)}
    \left ( \Asterisk_{0 \leq i < n_2, 1 \ot \Linfty(\rS^1)}(\ZZ \ltimes_{-1} \Linfty(\rS^1)) \right ) \\
    & \qquad *_{1 \ot \Linfty(\rS^1)}
    \left ( \Asterisk_{0 \leq i < m_0, 1 \ot \Linfty(\rS^1)}(\rL \ZZ \ot \Linfty(\rS^1)) \right ) \\
  & \cong
  (\rL \freegrp{n_1 + m_0} \vnt \Linfty(\rS^1)) *_{1 \ot \Linfty(\rS^1)} (\freegrp{n_1 + n_2} \ltimes_\beta \Linfty(\rS^1)) \\
  & =
  (\rL \freegrp{m} \vnt \Linfty(\rS^1)) *_{1 \ot \Linfty(\rS^1)} (\freegrp{n} \ltimes_\beta \Linfty(\rS^1))
\end{align*}
and this isomorphism carries $A_\pi = \rL \ZZ$ onto $\Linfty(\rS^1)$.
\end{proof}

\begin{corollary}
\label{cor:normal-ap}
  $A_\pi$ is regular inside $M_\pi$.
\end{corollary}
\begin{proof}
By Theorem \ref{thm:afp-decomposition-ap}, we know that
\[
  M_\pi
  \cong
  (\rL \freegrp{m} \vnt \Linfty(\rS^1)) *_{1 \ot \Linfty(\rS^1)} (\freegrp{n} \ltimes_\beta \Linfty(\rS^1))
  \eqcomma
\]
and $A_\pi$ is sent onto $1 \ot \Linfty (\rS^1)$ under this isomorphism.  It follows immediately that $A_\pi \subset M_\pi$ is regular.
\end{proof}

Note that in Theorem \ref{thm:afp-decomposition-ap} the action of $\freegrp{m}$ on $\rS^1$ is not free.

\begin{proposition}
\label{prop:relative-commutant-ap}
 Adopting the notation of Theorem \ref{thm:afp-decomposition-ap}, the relative commutant of $\Linfty(S^1)$ in $(\rL \freegrp{m} \ot \Linfty(\rS^1)) *_{1 \ot \Linfty(\rS^1)} (\freegrp{n} \ltimes \Linfty(\rS^1))$ is $\rL G \vnt \Linfty(\rS^1)$, where $G = \freegrp{m} * \ker \pi$ and $\pi:\freegrp{n} \ra \rS^1$ sends a generator $g_i$ to $\lambda_i$ and $h_i$ to $-1$.
\end{proposition}
\begin{proof}
  It is clear that the algebra generated by the elements $u_g$ with $g \in G$ is part of the relative commutant of $\Linfty(\rS^1)$ in $M_\pi$, so we have to prove the other inclusion.  Let $x \in \Linfty(\rS^1)' \cap M_\pi$ and write $x = \sum_{k \in Z} x_k u_k$ the Fourier decomposition with respect to the action of $\ZZ$ on $\Gamma(H_\pi)''$.  Then $x_k \in \rL \ZZ' \cap M_\pi$, so we can assume that $x \in \Gamma(H_\pi)'' \cong \rL (\freegrp{m + n})$. Write $x = \sum_{g \in \freegrp{m + n}}a_g u_g$ with $a_g \in \CC$. Since for all $g$ the action of $\alpha(1)$ leaves $\CC u_g$ invariant, $x$ is fixed by $\alpha$ if and only if it has only coefficients in $G$.  This proves the lemma.
\end{proof}

\begin{corollary}
\label{cor:factorial-ap}
  The von Neumann algebra $M_\pi$ is factorial if and only if $\pi$ is faithful.
\end{corollary}
\begin{proof}
  Let $\pi$ be a non-faithful representation and take $g \in \ZZ$ such that $\pi(g) = \id$.  Then $u_g \in \rL \ZZ$ is central in $M_\pi$.  For the converse implication, note that $\pi$ is faithful if and only if the eigenvalues of $\pi_\CC$ generate an infinite subgroup of $\rS^1$.  Any central element $x$ of $M_\pi$ must lie in $\rL G \vnt \rL \ZZ$ and hence in $\rL \ZZ$, since $G$ is a free group.  Writing $\rL \freegrp{n} \rtimes \ZZ \cong \freegrp{n} \ltimes \Linfty(\rS^1)$ as in Theorem \ref{thm:afp-decomposition-ap}, the assumption implies that the action of $\freegrp{n}$ on $\Linfty(\rS^1)$ is ergodic.  So $x \in \CC 1$.
\end{proof}

Using Proposition \ref{prop:relative-commutant-ap}, we can derive a representation of $M_\pi$ as a cocycle crossed product of $\rL G \vnt \rL \ZZ$ by the group $K \subset \rS^1$ generated by the eigenvalues of $\pi_\CC$.  For any element $k \in K$ choose an element $g_k \in \freegrp{n}$ such that $\alpha(1)u_{g_k} = k u_{g_k}$.  Define a $G$ valued $2$-cocycle $\Omega$ on $K$ by
\[\Omega(k,l) = g_{kl}g_l^{-1}g_k^{-1} \eqstop\]
Then $K$ acts on $G$ by conjugation and on $\rL \ZZ$ by $k * u_1 = k \cdot u_1$.
Note that if $K$ is cyclic and infinite, then we can choose $\Omega$ to be trivial.  In this case, denote by $g_1, g_2, \dotsc$ a basis of $\freegrp{m + n}$ such that $u_{g_1}$ acts by rotation on $\rS^1$ and $g_2, g_3, \dotsc$ commute with $A_\pi$.  We see that the elements $g_1^k g_i g_1^{-k}$, $i \geq 2$, $k \in \ZZ$ are a free basis of $G$.  So $K$ acts by shifting a free basis of $G$.  This proves the following proposition.

\begin{proposition}
\label{prop:cocycle-crossed-product}
  There is an isomorphism $(A_\pi \subset M_\pi) \cong (1 \otimes \Linfty(\rS^1) \subset K \ltimes_\Omega (\rL G \vnt \Linfty(\rS^1)))$. In particular, if $\pi$ is two dimensional and faithful, then $M_\pi \cong \ZZ \ltimes (\rL \freegrp{\infty} \vnt \Linfty(\rS^1))$, where $\ZZ$ acts on $\freegrp{\infty}$ by shifting the free basis and on $\rS^1$ by multiplication with a non-trivial eigenvalue of $\pi_\CC$.
\end{proposition}

\subsection{$A_\pi$-$A_\pi$-bimodules in $\Ltwo(M_\pi)$}
\label{sec:A-A-bimodules}

If $\pi$ is weakly mixing, it is known \cite[Proof of Theorem D.4]{vaes07} that every right finite $A_\pi$-$A_\pi$-bimodule is contained in $\Ltwo(A_\pi)$. More generally, we have the following proposition.

\begin{proposition}
\label{prop:right-finite-bimodules}
  Let $(\pi, H)$ be an orthogonal representation of $\ZZ$ and let $M_\pi = M_{\mathrm{ap}} *_{A_\pi} M_{\mathrm{wm}}$ be the decomposition of $M_\pi$ into almost periodic and weakly mixing part. Then every right finite $A_\pi$-$A_\pi$-bimodule in $\Ltwo(M_\pi)$ lies in $\Ltwo(M_{\mathrm{ap}})$.
\end{proposition}
\begin{proof}
By Lemma D.3 in \cite{vaes07}, we have to prove that there is a sequence of unitaries $(u_k)_k$ in $A$ tending to $0$ weakly such that for all $x,y \in M \ominus M_{\mathrm{ap}}$ we have $\| \rE_A(x u_n y^*) \|_2 \ra 0$.  It suffices to consider $x = w(\xi_1 \ot \dotsm \ot \xi_n), y = w(\eta_1 \ot \dotsm \ot \eta_m)$ for some $\xi_1 \ot \dotsm \ot \xi_n H^{\ot n}, \eta_1 \ot \dotsm \ot \eta_m \in H^{\ot m}$ such that at least one $\xi_i$ and one $\eta_j$ lie in $H_{\mathrm{wm}}$.  Take a sequence $(g_k)_k$ going to infinity in $\ZZ$ such that $\langle \pi(g_k)\xi, \eta \rangle \ra 0$ for all $\xi, \eta \in H_{\mathrm{wm}}$.  Then
\begin{align*}
  \| \rE_A(x u_{g_k} y^*) \|_2
  & =
  \| \rE_A(w(\xi_1 \ot \dotsm \xi_n)w(\pi(g_k)\eta_1 \ot \dotsm \ot \pi(g_k)\eta_m)^*)u_{g_k} \|_2 \\
  & =
  |\tau(w(\xi_1 \ot \dotsm \xi_n)w(\pi(g_k)\eta_1 \ot \dotsm \ot \pi(g_k)\eta_m)^*)| \\
  & =
  \langle \xi_1 \ot \dotsm \xi_n , \pi(g_k)\eta_1 \ot \dotsm \ot \pi(g_k)\eta_m) \rangle \\
  & =
  \delta_{n,m} \cdot \langle \xi_1 , \pi(g_k) \eta_1 \rangle \dotsm \langle \xi_n , \pi(g_k) \eta_n \rangle \\
  & \lra
  0
  \eqstop
\end{align*}
This finishes the proof.
\end{proof}

As an immediate consequence, we obtain the following corollaries.
\begin{corollary}
\label{cor:normaliser-quasi-normaliser}
  Let $\pi$ be an orthogonal representation of $\ZZ$.  The quasi-normaliser and the normaliser of $A_\pi \subset M_\pi$ are equal to $M_{\mathrm{ap}}$.  In particular, $A_\pi' \cap M_\pi = \rL G \vnt A_\pi$, where $G$ as defined in Proposition \ref{prop:relative-commutant-ap} is isomorphic to a free group.
\end{corollary}
\begin{proof}
  This follows from Proposition \ref{prop:right-finite-bimodules} and Corollary \ref{cor:normal-ap}.
\end{proof}

\begin{corollary}
\label{cor:factorial}
  If $\pi$ is an orthogonal representation of $\ZZ$, then $M_\pi$ is factorial if and only if $\pi$ is faithful.
\end{corollary}
\begin{proof}
  This follows from Proposition \ref{prop:right-finite-bimodules} and Corollary \ref{cor:factorial-ap}.
\end{proof}

\begin{remark}
  Note that Corollary \ref{cor:factorial} also follows directly from Theorem 5.1 of \cite{houdayershlyakhtenko11}.
\end{remark}


\section{Almost periodic representations}
\label{sec:flexibility-ap}

In this section, we prove that the isomorphism class of $M_\pi$ for an almost periodic orthogonal representation $\pi$ of the integers depends at most on the concrete subgroup of $\rS^1$ generated by the eigenvalues of the complexification of $\pi$.  We also classify non-faithful almost periodic orthogonal representations, that is periodic orthogonal representations, in terms of their kernel and their dimension.

\subsection{Isomorphism of free Bogoljubov crossed products of almost periodic representations depends at most on the subgroup generated by the eigenvalues of their complexifications}
\label{sec:normal-forms}

The following lemma will be used extensively in the proof of Theorem \ref{thm:dependence-on-subgroup-only}.

\begin{lemma}
\label{lem:change-free-basis}
  Let $S$ be any set and $x_s$, $s \in S$ a free basis of $\freegrp{S}$.  Let $I \subset S$ and $w_s$, $s \in I$ be words with letters in $\{x_s \amid s \in S \setminus I\}$.  Then $y_s = x_s w_s$, $s \in I$ together with $y_s = x_s$, $s \in S \setminus I$ form a basis of $\freegrp{S}$.
\end{lemma}
\begin{proof}
  It suffices to show that the map $\freegrp{S} \ra \freegrp{S}: x_s \mapsto y_s$ has an inverse.  This inverse is given by the map
  \[
    \freegrp{S} \ra \freegrp{S}:
    x_s \mapsto
    \begin{cases}
      x_sw_s^{-1}, & \text{if } s \in I \\
      x_s, & \text{otherwise.}
    \end{cases}
  \]
\end{proof}

\begin{theorem}
\label{thm:dependence-on-subgroup-only}
Let $\pi, \rho$ be orthogonal representations of $\ZZ$ whose almost periodic parts have the same dimension and the eigenvalues of their complexifications generate the same concrete subgroup of $\rS^1$.  Then $M_\pi \cong M_\rho$  via an isomorphism that is the identity on $A_\pi = \rL \ZZ = A_\rho$.
\end{theorem}
\begin{proof}
  By the amalgamated free product decomposition $M_\pi \cong M_{\mathrm{ap}} *_{A_\pi} M_{\mathrm{wm}}$ of Remark \ref{rem:decomposition-wm-ap}, it suffices to consider almost periodic representations.  Denote by $G$ the subgroup of $\rS^1$ generated by the eigenvalues of the complexification of $\pi$.
We may assume that the number of eigenvalues in $e^{2\pi i (0,\frac{1}{2})}$ of the complexification of $\pi$ is larger than the one of $\rho$.
Denote by $\lambda_i \in  e^{2\pi i (0,\frac{1}{2})}$, $0 \leq i < n_1$, $n_1 \in \NN \cup \{\infty\}$  and $\ol{\lambda_i}$, $0 \leq i < n_1$ the eigenvalues of the complexification of $\pi$ that are not equal to $1$ or $-1$.  Denote by $n_2, m_0 \in \NN \cup \{\infty\}$ the multiplicity of $-1$ and $1$, respectively, as eigenvalues of  $\pi$.  By Theorem \ref{thm:afp-decomposition-ap}, we have $M_\pi \cong \freegrp{\dim \pi} \ltimes \Linfty(\rS^1)$, where $\freegrp{\dim \pi}$ has a basis consisting of
\begin{itemize}
\item elements $x_i$, $0 \leq i < n_1$ acting on $\rS^1$ by multiplication with $\lambda_i$,
\item elements $y_i$, $0 \leq i < n_1$ acting trivially on $\rS^1$,
\item elements $z_i$, $0 \leq i < n_2$ acting on $\rS^1$ by multiplication with $-1$ and
\item elements $w_i$, $0 \leq i < m_0$ acting trivially on $\rS^1$.
\end{itemize}
Denote by $\mu_i \in e^{2 \pi i (0,\frac{1}{2})}$, $0 \leq i < l_1 \in \NN \cup \{ \infty \}$ the non-trivial eigenvalues of the complexification of $\rho$ that lie in the upper half of the circle and by $l_2,k_0 \in \NN \cup \{\infty\}$ the multiplicity of $-1$ and $1$, respectively, as an eigenvalue of $\rho$.  Since $\dim \pi = \dim \rho$, we have $2 \cdot l_1 + l_2 + k_0 = 2 \cdot n_1 + n_2 + m_0$.  We will find a new basis $r_i$ $(0 \leq i < l_1)$, $s_i$ $(0 \leq i < l_1 + k_0)$, $t_i$ $(0 \leq i < l_2)$ of $\freegrp{\dim \pi}$ such that
\begin{itemize}
\item $r_i$, $0 \leq i < l_1$, acts by multiplication with $\mu_i$ on $\rS^1$,
\item $s_i$, $0 \leq i < k_0 + l_1$, acts trivially on $\rS^1$ and
\item $t_i$, $0 \leq i < l_2$, acts by multiplication with $-1$ on $\rS^1$.
\end{itemize}
Invoking Theorem \ref{thm:afp-decomposition-ap}, this suffices to finish the proof.

In what follows, we will apply Lemma \ref{lem:change-free-basis} repeatedly.  Replace the basis elements $y_i$, $0 \leq i < n_1$ by $\tilde y_i = y_i x_i$ for $0 \leq i < n_1$.  Then $\tilde y_i$ acts on $\rS^1$ by multiplication with $\mu_i$, $0 \leq i < n_1$.  Recall that we assumed $l_1 \leq n_1$.  Since the subgroups of $\rS^1$ generated by the eigenvalues of the complexifications of $\pi$ and $\rho$ agree, for every $0 \leq i < l_1$ there are elements $a_{i,1} \dotsc, a_{i,\alpha} \in \ZZ$, $0 \leq j_{i,1}, \dotsc, j_{i,\alpha} < n_1$ and $a_{i,0} \in \{0,1\}$ such that
\[
  \mu_i
  =
  \lambda_{j_1}^{a_{i,1}} \dotsm \, \lambda_{j_\alpha}^{a_{i,\alpha}} \cdot (-1)^{a_{i,0}}
  \eqcomma
\]
where $a_{i,0} = 0$ if $-1$ is not an eigenvalue of $\pi$.  Replacing $x_i$, $0 \leq i < l_1$ by
\[
  r_i
  =
  x_i \tilde y_i^{-1} \tilde y_{j_{i,1}}^{a_{i,1}} \dotsm \, \tilde y_{j_{i, \alpha(i)}}^{a_{i,\alpha(i)}} \cdot z_1^{a_{i,0}}
  \eqcomma
\]
we obtain a new basis of $\freegrp{\dim \pi}$ consisting of $r_i$ $(0 \leq i <l_1)$, $x_i$ $(l_1 \leq i < n_1)$, $\tilde y_i$ $(0 \leq i < n_1)$, $z_i$ $(0 \leq i < n_2)$ and $w_i$ $(0 \leq i < m_0)$.

We distinguish whether $-1$ in an eigenvalue of $\rho$ or not.  If $-1$ is no eigenvalue of $\rho$, we produce elements $s_i$ $(0 \leq i < (n_1 - l_1) + n_1 + n_2 + m_0 )$ acting trivially on $\rS^1$, where we put $n_1 - l_1 = 0$, if $l_1 = n_1 = \infty$.  Replace $x_i$ by $x_i \tilde y_i^{-1}$ for $l_1 \leq i < n_1$ and then multiply $\tilde y_i$, $0 \leq i < n_1$  and $z_i$, $0 \leq i < n_2$ from the right with words in $r_i$, $0 \leq i < l_1$ so as to obtain these new basis elements $s_i$ $(0 \leq i < (n_1 - l_1) + n_1 + n_2 + m_0 )$.  Since $\dim \pi = 2n_1 + n_2 + m_0 = l_1 + (n_1 - l_1) + n_1 + n_2 + m_0$ and $l_2 = 0$, we found a basis $r_i$ $(0 \leq i < l_1)$, $s_i$ $(0 \leq i < l_1 + k_0)$ of $\freegrp{\dim \pi}$ acting on $\rS^1$ as desired.  This finishes the proof in the case $-1$ is no eigenvalue of $\rho$.

Now assume that $-1$ is an eigenvalue of $\rho$.  We distinguish three further cases. \\
\textit{Case $l_1 < n_1$:}  There are elements $a_1 \dotsc, a_\alpha \in \ZZ$, $0 \leq i_1, \dotsc, i_\alpha < n_1$ and $a_0 \in \{0,1\}$ such that
\[
  -1
  =
  \lambda_{i_1}^{a_1} \dotsm \, \lambda_{i_\alpha}^{a_\alpha} \cdot (-1)^{a_0}
  \eqcomma
\]
where $a_0 = 0$ if $-1$ is not an eigenvalue of $\pi$.  Replace $x_{l_1 + 1}$ by
\[
  t_1
  =
  x_{l_1 + 1} \tilde y_{l_1 + 1}^{-1} \tilde y_{i_1}^{a_1} \dotsm \, \tilde y_{i_\alpha}^{a_\alpha} z_1^{a_0}
  \eqstop
\] \\
\textit{Case $l_1 = n_1$ and $-1$ is an eigenvalue of $\pi$:} Put $t_1 = z_1$. \\
\textit{Case $l_1 = n_1$ and $-1$ is no eigenvalue of $\pi$:} Since $2n_1 + m_0 = 2l_1 + l_2 + k_0$, in this case, $\pi$ has a trivial subrepresentation of dimension $1$ or $\pi$ is infinite dimensional. Hence, we may assume that $m \geq 1$, since all $y_i, 0 \leq i < n_1$ act trivially on $\rS^1$.  There are elements $a_1 \dotsc, a_\alpha \in \ZZ$, $0 \leq i_1, \dotsc, i_\alpha < n_1$ such that
\[
  -1
  =
  \lambda_{i_1}^{a_1} \dotsm \, \lambda_{i_\alpha}^{a_\alpha}
  \eqstop
\]
Put
\[
  t_1
  =
  w_1 \tilde y_{i_1}^{a_1} \dotsm \, \tilde y_{i_\alpha}^{a_\alpha}
  \eqstop
\]

In all three cases, we obtain a basis of $\freegrp{\dim \pi}$ with elements $r_i$ $(0 \leq i < l_1)$, possibly $t_1$ and some other elements such that
\begin{itemize}
\item $r_i$, $0 \leq i < l_1$, acts by multiplication with $\mu_i$ on $\rS^1$,
\item $t_1$ acts by multiplication with $-1$ on $\rS^1$ and
\item all other elements of the basis act on $\rS^1$ by multiplication with some element in $G \subset \rS^1$.
\end{itemize}
We can multiply the elements different from $r_i$, $(0 \leq i < l_1)$, and $t_1$ in the basis by some word in the letters $r_i$, $0 \leq i < l_1$ and $t_1$ in order to obtain a basis $r_i$ $(0 \leq i < l_1)$, $s_i$ $(0 \leq i < \dim \pi - l_1 -1)$, $t_1$ or $r_i$ $(0 \leq i < l_1)$, $s_i$ $(0 \leq i < \dim \pi - l_1)$ where all elements $s_i$ act trivially on $\rS^1$.  We used the convention $\dim \pi - l_1 = \infty$, if $l_1 = \dim \pi = \infty$.
If $l_1 + k_0 < \infty$, replace $s_i$, $(l_1 + k_0 \leq i < l_1 + k_0 + l_2 - 1)$ by $t_{i - k + 2} =  s_i \cdot t_1$, in order to obtain a basis $r_i$ $(0 \leq i < l_1)$, $s_i$ $(0 \leq i < l_1 + k_0)$, $t_i$ $(0 \leq i < l_2)$ of $\freegrp{\dim \pi}$ acting on $\rS^1$ as desired.  If $l_1 + k_0 = \infty$, then replace $l_2$-many $s_i$ by $s_i t_1$ so as to obtain the new basis $r_i$ $(0 \leq i < l_1)$,  $s_i$ $(0 \leq i < l_1 + k_0)$, $t_i$ $(0 \leq i < l_2)$ of $\freegrp{\dim \pi}$ acting on $\rS^1$ as desired.  This finished the proof.
\end{proof}

\subsection{The classification of free Bogoljubov crossed products associated with periodic representations of the integers is equivalent to the isomorphism problem for free group factors}
\label{sec:periodic}

The classification of free Bogoljubov crossed products associated with non-faithful, that is periodic, orthogonal representations of $\ZZ$ implies a solution to the isomorphism problem for free group factors.  For example, if $\mathbb{1}$ denotes the trivial orthogonal representation of $\ZZ$, we have $M_{n \cdot \mathbb{1}} \cong \rL \freegrp{n} \vnt \rL \ZZ$.  So, proving whether $M_{n \cdot \mathbb{1}} \cong M_{m \cdot \mathbb{1}}$ or not for different $n$ and $m$ amounts to solving the isomorphism problem for free group factors. More generally, we have the following result.

\begin{theorem}
\label{thm:classification-periodic}
  Let $\pi$ be a periodic orthogonal representation of the integers.  If $\pi$ is trivial, then $A_\pi \subset M_\pi$ is isomorphic to an inclusion $1 \ot \Linfty([0,1]) \subset \rL \freegrp{\dim \pi} \ot \Linfty([0,1])$.  If $\pi$ is one dimensional and non-trivial, then $(A_\pi \subset M_\pi) \cong (\CC^2 \ot 1 \ot \Linfty([0,1]) \subset \Cmat{2} \ot \Linfty([0,1]) \ot \Linfty([0,1]))$.  If $\pi$ has dimension at least $2$, let $T$ be the index of the kernel of $\pi$ in $\ZZ$.  Then $(A_\pi \subset M_\pi) \cong (\CC^T \otimes \Linfty([0,1]) \subset \rL \freegrp{r} \vnt \Linfty([0,1]))$, where $\rL \freegrp{r}$ is an interpolated free group factor with parameter
\[r = 1 + \frac{1}{T}(\dim \pi - 1) \eqstop\]
\end{theorem}
\begin{proof}
  The case where $\pi$ is trivial, follows directly from the definition of $\Gamma(H, \ZZ, \pi)''$.  To prove all other cases, by Theorem \ref{thm:dependence-on-subgroup-only}, it suffices to consider representations $\pi = \pi_0 \oplus n \cdot \mathbb{1}$ with $\pi_0$ irreducible and non-trivial and $n \in \NN \cup \{\infty\}$. 

We first consider irreducible representations.  The case of $\pi$ one dimensional is immediately verified from the definition of $M_\pi = \Gamma(H, \ZZ, \pi)''$.  If $\pi$ has dimension $2$ and is irreducible denote by $\lambda = e^{\frac{2\pi i}{T}}$ and $\ol{\lambda} = e^{-\frac{2\pi i}{T}}$, with $T = [\ZZ : \ker \pi] \in \NN_{\geq 2}$, the eigenvalues of $\pi_\CC$.  Then
\begin{align*}
  M_\pi
  & \cong
  (\rL \ZZ \vnt \rL \ZZ) *_{1 \otimes \rL \ZZ} (\rL \ZZ \rtimes_{\lambda} \ZZ) \\
  & \cong
  (\rL \ZZ \vnt \CC^T \vnt \Linfty([0,1])) *_{1 \otimes \CC^T \vnt \Linfty([0,1])} (\Linfty([0,1]) \vnt \Cmat{T} \vnt \Linfty([0,1])) \\
  & \cong
  ((\rL \ZZ \vnt \CC^T) *_{1 \otimes \CC^T} (\Linfty([0,1]) \vnt \Cmat{T})) \vnt \Linfty([0,1])
  \eqstop
\end{align*}
Since $(\rL \ZZ \vnt \CC^T) *_{1 \otimes \CC^T} (\Linfty([0,1]) \vnt \Cmat{T})$ is a non-amenable factor by Theorem \ref{thm:houdayervaes}, Theorem \ref{thm:dykemaredelmeier} shows that
\[((\rL \ZZ \vnt \CC^T) *_{1 \otimes \CC^T} (\Linfty([0,1]) \vnt \Cmat{T})) \cong \rL \freegrp{r}\]
with
\[
  r 
  = 1 + 1 - (1 - \frac{1}{T})
  = 1 + \frac{1}{T}(\dim \pi - 1) \eqstop
\]
Moreover,
\begin{align*}
  (A_\pi \subset M_\pi) 
  & \cong (1 \ot \CC^T \vnt \Linfty([0,1]) \subset ((\rL \ZZ \vnt \CC^T) *_{1 \otimes \CC^T} (\Linfty([0,1]) \vnt \Cmat{T})) \vnt \Linfty([0,1])) \\
  & \cong (\CC^T \otimes \Linfty([0,1]) \subset \rL \freegrp{r} \vnt \Linfty([0,1])) \eqstop
\end{align*}

Consider now $\pi = \pi_0 \oplus n \cdot \mathbb{1}$ for an irreducible, non-trivial and non-faithful representation of dimension two $\pi_0$. The case where $\pi_0$ is of dimension one and has eigenvalue $-1$ is similar, but simpler.  Let $T = [\ZZ : \ker \pi_0] \in \NN_{\geq 2}$ and $n \in \NN \cup \{\infty\}$.  Let $r_0 = 1 + \frac{1}{T}$. Then Theorems \ref{thm:dykemaredelmeier} and \ref{thm:houdayervaes} imply that
\begin{align*}
  M_{\pi_0 \oplus n \cdot \tau}
  & \cong
  (\rL \freegrp{n} \vnt \rL \ZZ) *_{1 \ot \rL \ZZ \cong \CC^T \ot \Linfty([0,1])} (\rL \freegrp{r_0} \vnt \Linfty([0,1])) \\
  & \cong
  (\rL \freegrp{n} \ot \CC^T *_{1 \ot \CC^T} \rL \freegrp{r_0}) \vnt \Linfty([0,1]) \\
  & \cong
  \rL \freegrp{r} \ot \Linfty([0,1])
  \eqcomma
\end{align*}
with
\[r = 1 + \frac{1}{T}(n - 1) + r_0 - (1 - \frac{1}{T}) = \frac{1}{T}(\dim(\pi_0 \oplus n \cdot \mathbb{1}) - 1) \eqstop\]
Also
\[(A_\pi \subset M_\pi)  \cong (\CC^T \otimes \Linfty([0,1]) \subset \rL \freegrp{r} \vnt \Linfty([0,1]))\]
and this finishes the proof.
\end{proof}

\subsection{A flexibility result for representations with one pair of non-trivial eigenvalue}
\label{sec:one-irrational-eigenvalue}

In this section, we will show that all free Bogoljubov crossed products associated with almost periodic orthogonal representations of $\ZZ$ with a single non-trivial irreducible component, which is faithful, are isomorphic.

\begin{proposition}
\label{prop:one-irrational-eigenvalue}
Let $\pi_i$ for $i \in \{1,2\}$ be almost periodic orthogonal representations of $\ZZ$ having the same dimension.  Assume that their complexifications $(\pi_i)_\CC$ each have a single pair of non-trivial eigenvalues $\lambda_i, \ol{\lambda_i} \in e^{2\pi i \RR \setminus \QQ}$ with any multiplicity.  Then $M_{\pi_1} \cong M_{\pi_2}$ by an isomorphism, which carries $A_{\pi_1}$ onto $A_{\pi_2}$.
\end{proposition}
\begin{proof}
  By Theorem \ref{thm:dependence-on-subgroup-only} is suffices to consider the case where the eigenvalue $\lambda_i$ of $(\pi_i)_\CC$ has multiplicity one.  Theorem \ref{thm:afp-decomposition-ap} shows that
\[
  M_{\pi_1}
  \cong
  (\rL \freegrp{\dim \pi_1 - 1} \ot \Linfty(\rS^1)) *_{1 \ot \Linfty(\rS^1)} (\ZZ \ltimes_{\lambda_i} \Linfty(\rS^1))
  \eqcomma
\]
by an isomorphism, which caries $A_{\pi_i}$ onto $\Linfty(\rS^1)$.  Taking an orbit equivalence of the ergodic hyperfinite $II_1$ equivalence relations induced by $\ZZ \grpaction{\lambda_1} \rS^1$ and $\ZZ \grpaction{\lambda_2} \rS^1$, we obtain an isomorphism $\ZZ \ltimes_{\lambda_1} \Linfty(\rS^1) \cong \ZZ \ltimes_{\lambda_2} \Linfty(\rS^1)$, which preserves $\Linfty(\rS^1)$ globally.  This can be extended to an isomorphism $M_{\pi_1} \cong M_{\pi_2}$, which carries $A_{\pi_1}$ onto $A_{\pi_2}$.
\end{proof}

\begin{corollary}
\label{cor:faithful-two-dimensional}
  All faithful two dimensional representations of $\ZZ$ give rise to isomorphic free Bogoljubov crossed products.
\end{corollary}

\subsection{Some remarks on a possible classification of Bogoljubov crossed products associated with almost periodic orthogonal representations}
\label{sec:classificaton}

In Theorem \ref{thm:dependence-on-subgroup-only} we showed that the isomorphism class of free Bogoljubov crossed products associated with almost periodic orthogonal representations of $\ZZ$ depends at most on the concrete subgroup of $\rS^1$ generated by the eigenvalues of its complexification.  However, Theorem \ref{thm:classification-periodic} and Proposition \ref{prop:one-irrational-eigenvalue} both show that there are orthogonal representations $\pi, \rho$ of $\ZZ$ such that these subgroups of $\rS^1$ are not equal and still they give rise to isomorphic free Bogoljubov crossed products.  This answers a question of Shlyakhtenko, asking whether a complete invariant for the isomorphism class of the free Bogoljubov crossed products associated with an orthogonal representation $\pi$ of $\ZZ$ is $\oplus_{n \geq 1}\pi^{\otimes n}$ up to amplification.  By Theorem \ref{thm:classification-periodic}, the classification of free Bogoljubov crossed products associated with non-faithful orthogonal representations of $\ZZ$ is equivalent to the isomorphism problem for free group factors.  However, assuming that $M_\pi$ is a factor, i.e. that $\pi$ is faithful, the abstract isomorphism class of the group generated by the eigenvalues of the complexification of $\pi$ could be an invariant.  Due to the fact that the isomorphisms found in Theorem \ref{thm:classification-periodic} preserve the subalgebra $A_\pi \subset M_\pi$ for non-faithful orthogonal representations, we believe that this abstract group is indeed an invariant for infinite dimensional representations.

\begin{conjecture}
\label{conj:invariant-for-ap}
The abstract isomorphism class of the subgroup generated by the eigenvalues of the complexification of an infinite dimensional faithful almost periodic orthogonal representation of $\ZZ$ is a complete invariant for isomorphism of the associated free Bogoljubov crossed product.
\end{conjecture}


\section{Solidity and strong solidity for free Bogoljubov crossed products}
\label{sec:solidity}

The proof of the following result can be extracted literally from the proof of \cite[Theorem 1]{shlyakhtenko03-multiplicity}.  It shows that the dimension of the almost periodic part of an orthogonal representation of $\ZZ$ is relevant for the isomorphism class of its free Bogoljubov crossed product.  We give a full prove for the convenience of the reader.  Recall that we denote by $\mathbb{1}$ the trivial orthogonal representation of the integers.

\begin{theorem}
\label{thm:flexibility-for-left-regular}
  The free Bogoljubov crossed products $M_\lambda$ and $M_{\lambda \oplus \mathbb{1}}$ are isomorphic to $\rL \freegrp{2}$.
\end{theorem}
\begin{proof}
  We have $M_\lambda \cong \rL \freegrp{\infty} \rtimes \ZZ$, where $\ZZ$ acts by shifting a free basis of $\freegrp{\infty}$, so $M_\lambda \cong \rL \freegrp{2}$.  Consider $M_{\lambda \oplus \mathbb{1}} \cong M_\lambda *_A (\rL \ZZ \vnt A)$.  Let $B = \rL \ZZ \vnt A$.  By \cite{shlyakhtenko98-applications}, we know that $M_{\lambda}$ is isomorphic to the free Krieger algebra $\Phi(A, \tau)$ for the completely positive map $\tau: A \ra \CC \subset A$.  Let $X \in M_{\lambda}$ be the $A$-valued semicircular variable coming from this isomorphism.  We show that $X$ is $B$-valued semicircular with distribution $\tau_B = \tau \ot \tau: B \ra \CC \subset B$.  Then it follows that $M_{\lambda \oplus \mathbb{1}} \cong \Phi(B, \tau_B) \cong \rL \freegrp{2}$.

From the definition of freeness with a amalgamation, we see that for all $b_1, \dotsc, b_n \in B$ we have
\[
  \rE_B(Xb_1X \dotsm b_nX)
  = 
  \rE_A(X\rE_A(b_1)X \dotsm \rE_A(b_n)X)
  =
  \rE_A(X (\id \ot \tau)(b_1) X \dotsm (\id \ot \tau)(b_n) X)
  \eqstop
\]
As a result, for the free cumulants of $c_{(X,B)}^{(n)}$ of $X$ with respect to $B$ can be expressed in terms of the free cumulants $c_{(X,A)}^{(n)}$ of $X$ with respect to $A$ as
\[
  c_{(X,B)}^{(n)}(b_1, \dotsc, b_n)
  =
  c_{(X,A)}^{(n)}((\id \ot \tau)(b_1), \dotsc, (\id \ot \tau)(b_n))
  =
  \begin{cases}
    (\tau \ot \tau)(b_1) & \text{ if } n = 1 \eqcomma\\
    0 & \text{ otherwise} \eqstop
  \end{cases}
\]
This shows that $X$ is $B$-valued semicircular with distribution $\tau_B = \tau \ot \tau: \rL \ZZ \vnt A \ra \CC \subset \rL \ZZ \vnt A$.  We have finished the proof.
\end{proof}

The fact that the left regular representation plus a trivial one dimensional representation gives rise to a strongly solid free Bogoljubov crossed product, triggered the following observation.
\begin{theorem}
\label{thm:free-bogoliubov-strong-solidity}
  Let $\pi$ be an orthogonal representation of $\ZZ$ that is the direct sum of a mixing representation and a representation of dimension at most one.  Then $M_\pi$ is strongly solid.
\end{theorem}

This theorem follows from the next, more general, one. Its proof can be taken almost literally from \cite[Theorem 1.8]{ioana12}.  We include a proof for the convenience of the reader.

\begin{theorem}
\label{thm:strong-solidity}
  Let $A \subset N$ be a mixing inclusion of $A$ into a strongly solid, non-amenable, tracial von Neumann algebra.  Let $A \subset B$ an inclusion of $A$ into an amenable, tracial von Neumann algebra.  Then $M = N *_A B$ is strongly solid.
\end{theorem}
\begin{proof}
  We first show that $B \subset M$ is mixing.  As in \cite[Theorem 1.8]{ioana12}, we have to show that for every sequence $(b_n)_n$ in $(B)_1$ with $b_n \ra 0$ weakly and for all $a,b \in B$, $x, y  \in N \ominus A$ we have
\[\rE_A(x\rE_A(ab_nb)y) \stackrel{\twonorm}{\lra} 0 \eqstop \]
Since $b_n \ra 0$ weakly, also $\rE_A(ab_nb) \ra 0$ weakly.  The fact that $A \subset N$ is mixing, then implies that $\| \rE_A(x\rE_A(ab_nb)y) \|_2 \ra 0$.

Let $Q \subset M$ be a diffuse, amenable von Neumann subalgebra and write $P = \Norm_M(Q)''$.  Let $p \in \cZ(P)$ be the maximal projection such that $Pp$ has no amenable direct summand.  We assume $p \neq 0$ and deduce a contradiction.  Let $\omega$ be a non-principal ultrafilter.  By Theorem \ref{lem:ioana-relative-commutant-in-ultrapower} we have $p = e + f$ with $e,f \in \cZ((Pp)' \cap pMp) \cap \cZ((Pp)' \cap (pMp)^\omega$) such that
\begin{itemize}
\item $e((Pp)' \cap (pMp)^\omega) = e((Pp)' \cap pMp)$ and this algebra is atomic and
\item $f((Pp)' \cap (pMp)^\omega)$ is diffuse. 
\end{itemize}
Either $e \neq 0$ or $f \neq 0$.  In both cases, we will deduce that $Pp \prec_M N$.

If $e \neq 0$ let $e_0 \in (Pp)' \cap pMp$ be a minimal projection. Then $(Pe_0)' \cap (e_0Me_0)^\omega = \CC e_0$, so Theorem \ref{thm:ioana-main} applies to $Ae_0 \subset e_0Me_0$ and $Pe_0 \subset \Norm_{e_0Me_0}(Qe_0)''$.  We obtain that one of the following holds.
\begin{itemize}
\item $Qe_0 \prec_M A$,
\item $Pe_0 \prec_M N$,
\item $Pe_0 \prec_M B$ or
\item $Pe_0$ is amenable relative to $A$.
\end{itemize}
The first item implies that $Qe_0 \prec B$ and since $B \subset M$ is mixing, Lemma \ref{lem:ioana-control-normalizer-mixing} shows that $Pe_0 \prec B$.  So the first and the last two items imply that $Pe_0$ has an amenable direct summand, which contradicts the choice of $p$.  We obtain $Pp \prec_M N$ in the case $e \neq 0$.

If $f \neq 0$ then Theorem \ref{thm:ioana-control-relative-commutant} applied to $Pf \subset fMf$ shows that one of the following holds.
\begin{itemize}
\item $(Pf)' \cap (fMf)^\omega \prec_{M^\omega} A^\omega$,
\item $Pf \prec_M N$,
\item $Pf \prec_M B$ or
\item there is a non-zero projection $f_0 \in \cZ((Pf)' \cap fMf)$ such that $Pf_0$ is amenable relative to $A$. 
\end{itemize}
The first item implies $(Pf)' \cap (fMf)^\omega \prec_{M^\omega} B^\omega$ and since $B \subset M$ is mixing, Theorem \ref{thm:approximate-commutators-and-mixing-subaglebras} shows that $Pf \prec_M B$.  So the first and the last two items imply that $Pf$ has an amenable direct summand, contradicting the choice of $p$.  This shows $Pp \prec_M N$ in the case $f \neq 0$.

We showed $Pp \prec_M N$.  Let $p_0 \in P$, $q \in Q$, $p_0 \leq p$ be non-zero projections, $v \in pMq$ satisfying $vv^* = p_0$ and $\phi: p_0Pp_0 \ra qNq$ a *-homomorphism such that $xv = v\phi(x)$ for all $x \in p_0Pp_0$.  We have $v^*v \in \phi(p_0Pp_0)' \cap M$. Since $p_0 P p_0$ has no amenable direct summand it follows that $\phi(p_0Pp_0) \not \prec_M A$, and hence Theorem \ref{thm:ipp} shows that $v^*v \in N$.  So we can conjugate $P$ by a unitary in order to assume $p_0Pp_0 \subset N$.  Take partial isometries $w_1, \dotsc, w_n \in P$ such that $z = \sum_i w_iw_i^* \in \cZ(P)$ and $w_i^*w_i = \tilde{p} \leq p_0$ for all $i = 1, \dotsc, n$.  Then we obtain a *-homomorphism
\[
  \psi:
  Pz \ra \Cmat{n} \ot \tilde{p}N\tilde{p}:
  x \mapsto (w_i^* x w_j)_{i,j}
  \eqstop
\]
By \cite[Proposition 5.2]{houdayer10-non-free-gropup}, know that $\Cmat{n} \ot \tilde{p}N\tilde{p}$ is strongly solid.  This contradicts
\[
  \psi(Pz)
  \subset
  \Norm_{\Cmat{n} \ot \tilde{p}N\tilde{p}}(\psi(Az))''
\]
and the choice of $p$.
\end{proof}

\begin{proof}[Proof of Theorem \ref{thm:free-bogoliubov-strong-solidity}]
  Write $\pi = \pi_1 \oplus \pi_2$ with $\pi_1$ mixing and $\dim \pi_2 \leq 1$.  Then $M_\pi \cong M_{\pi_1} *_A M_{\pi_2}$.  Since $A \subset M_{\pi_1}$ is mixing by \cite[Theorem D.4]{vaes07}, it is strongly solid by \cite[Theorem B]{houdayershlyakhtenko11}.  Also $M_{\pi_2}$ is amenable, so Theorem \ref{thm:strong-solidity} applies.
\end{proof}

We have a partial converse to the previous theorem.
\begin{theorem}
\label{thm:non-solidity}
  Let $\pi$ be an orthogonal representation of $\ZZ$ with a rigid subspace of dimension at least two.  Then $M_\pi$ is not solid.
\end{theorem}
\begin{proof}
    Let $\omega$ be a non-principal ultrafilter.  Let $\xi, \eta \in H$ be orthogonal vectors such that there is a sequence $(n_k)_k$ going to infinity in $\ZZ$ and $\pi(n_k) \xi \ra \xi$, $\pi(n_k) \eta \ra \eta$ if $k \ra \infty$.  Then $[u_{n_k}] \in A^\omega$ is a Haar unitary and hence $P = \{s(\xi), s(\eta)\}''$ is a non-amenable subalgebra such that $P' \cap A^\omega \subset P' \cap M_\pi^\omega$ is diffuse.  Applying \cite[Proposition 7]{ozawa04-solid} to $P \subset M_\pi$ shows that $M_\pi$ is not solid.
\end{proof}

We conjecture that the previous theorem is sharp.
\begin{conjecture}
\label{conj:solidity}
  Let $\pi$ be an orthogonal representation of $\ZZ$.  Then the following are equivalent.
  \begin{itemize}
  \item $M_\pi$ is strongly solid.
  \item $M_\pi$ is solid.
  \item $\pi$ has no rigid subspace of dimension two.
  \end{itemize}
\end{conjecture}

The Theorems \ref{thm:free-bogoliubov-strong-solidity} and \ref{thm:non-solidity} of this work as well as Theorem A of \cite{houdayer12-structure} on free Bogoljubov crossed products that do not have property Gamma are supporting evidence for our conjecture.  We explain how Houdayer's result is related it.

\begin{theorem}[See Theorem A of \cite{houdayer12-structure}]
Let $G$ be a countable discrete group and $\pi: G \ra \cO(H)$ any faithful orthogonal representation such that $\dim H \geq 2$ and $\pi(G)$ is discrete in $\cO(H)$ with respect to the strong topology. Then $\Gamma(H)'' \rtimes_{\pi} G$ is a II$_1$ factor which does not have property Gamma.
\end{theorem}

First of all, note that in view of Proposition 7 of \cite{ozawa04-solid}, being non-Gamma can be considered as a weak form of solidity.  Secondly, we remark that an orthogonal representation $\pi:G \ra \cO(H)$ has discrete range, if and only if the whole Hilbert space $H$ is not rigid in our terminology.  This explains the link between our conjecture and the result of Houdayer.


\section{Rigidity results}
\label{sec:rigidity}

In this section, we want to show how to extract some information about $\pi$ from the von Neumann algebra $M_\pi$. As an application, we exhibit orthogonal representations of $\ZZ$ that cannot give rise to isomorphic free Bogoljubov crossed products.

\begin{theorem}
\label{thm:intertwining}
Let $\pi_1$, $\pi_2$ be orthogonal representations of $\ZZ$ such that each of them has a finite dimensional invariant subspace of dimension $2$. Assume that $M = M_{\pi_1} \cong M_{\pi_2}$. Let $A = A_{\pi_1}$ and identify $A_{\pi_2}$ with a subalgebra $B \subset M$. Then there is a finite index $A$-$B$-subbimodule of $\Ltwo(M)$.
\end{theorem}
\begin{proof}
  We want to use Theorem \ref{prop:fi-bimodule} in order to find a finite index $A$-$B$ bimodule in $\Ltwo(M)$.  So we have to verify its assumptions.  Corollary \ref{cor:normaliser-quasi-normaliser} implies that the normalisers of $A$ and $B$ are non-amenable.  So by Corollary \ref{cor:full-embedding}, $A \prec_M^\rmf B$ and $B \prec_M^\rmf A$ hold.  By Proposition \ref{prop:right-finite-bimodules}, every right finite $A$-$A$ subbimodule of $\Ltwo(M)$ lies in $\Ltwo(\QN_M(A)'')$.  So Theorem \ref{prop:fi-bimodule} says that there is a finite index $A$-$B$-subbimodule of $\Ltwo(M)$.
\end{proof}

\begin{corollary}
\label{cor:isomorphic-bimodules}
  Let $\pi_1, \pi_2$ be two orthogonal representations of $\ZZ$ having a finite dimensional subrepresentation of dimension at least $2$. Let $A_1 \subset M_1$ and $A_2 \subset M_2$ be the inclusions of the free Bogoljubov crossed products associated with $\pi_1$ and $\pi_2$, respectively.  Assume that $M_1 \cong M_2$.  Then there are projections $p_1 \in A_1$, $p_2 \in A_2$ and an isomorphism $\phi:A_1p_1 \ra A_2p_2$ preserving the normalised traces such that the bimodules $\bim{A_1p_1}{(p_1\Ltwo(M)p_1)}{A_1p_1}$ and $\bim{\phi(A_1p_1)}{(p_2\Ltwo(M)p_2)}{\phi(A_1p_1)}$ are isomorphic.
\end{corollary}
\begin{proof}
 By Theorem \ref{thm:intertwining}, there are projections $p_1 \in A_1$, $p_2 \in A_2$, an isomorphism $\phi:A_1p_1 \ra A_2p_2$ and a partial isometry $v \in p_1 M p_2$ such that $av = v\phi(a)$ for all $a \in A_1p_1$. Denote by $q_1$ and $q_2$ the left and right support of $v$, respectively. Cutting down $p_1$ and $p_2$, we can assume that $\supp \rE_{A_1}(q_1) = p_1$ and  $\supp \rE_{A_2}(q_2) = p_2$.  The bimodules $\bim{A_1p_1}{(q_1\Ltwo(M)q_1)}{A_1p_1}$ and $\bim{\phi(A_2p_2)}{(q_2\Ltwo(M)q_2)}{\phi(A_2p_2)}$ are isomorphic.

Since $p_1$ is the central support of $q_1$ in $A' \cap M$, there are projections $e_n \in A$, $n \in \NN$ such that $q_1 = \sum_n e_n$ and partial isometries $v_n^k \in A' \cap M$, $n \in \NN$, $k \leq n$ such that $\sum_k v_n^k(v_n^k)^* = e_n$ and $(v_n^1)^*v_n^1 = e_n q_1$, $(v_n^k)^*v_n^k \leq q_1$, for all $n$ and all $2 \leq k \leq n$.  Since the multiplicity function of $\bim{A_1}{\Ltwo(M)}{A_1}$ is constantly equal to infinity by Proposition \ref{prop:spectral-decomposition-FBCP}, we find that
\[
  \bim{Ae_n}{(e_nq_1 \Ltwo(M) e_n q_1)}{Ae_n}
  \cong
  \bigoplus_{k \leq n}
    \bim{Ae_n}{(v_n^k \Ltwo(M) (v_n^k)^*)}{Ae_n}
  \cong
  \bim{Ae_n}{(e_n \Ltwo(M) e_n)}{Ae_n}
  \eqcomma
\]
for all $n$.  So also
\[
  \bim{Ap_1}{(p_1 \Ltwo(M) p_1)}{Ap_1}
  \cong
  \bim{Ap_1}{(q_1 \Ltwo(M) q_1)}{Ap_1}
  \eqstop
\]
Similarly, we have  $\bim{A_2p_2}{(p_2\Ltwo(M)p_2)}{A_2p_2} \cong \bim{A_2p_2}{(q_2\Ltwo(M)q_2)}{A_2p_2}$.  This finishes the proof.
\end{proof}

A measure theoretic reformulation of Corollary \ref{cor:isomorphic-bimodules} can be given as follows.

\begin{corollary}
\label{cor:measure-class-equal}
  Let $(\mu_1, N_1), (\mu_2, N_2)$ be symmetric probability measures with multiplicity function on $\rS^1$ such that both have at least $2$ atoms when counted with multiplicity.  For $i = 1,2$, let $\pi_i$ be the orthogonal representation of $\ZZ$ by multiplication with $\id_{\rS^1}$ on $\Ltwo_\RR(\rS^1, \mu_i, N_i)$.  If $M_{\pi_1} \cong M_{\pi_2}$, then there are Lebesgue non-negligible Borel subsets $B_1, B_2 \subset \rS^1$ and a Borel isomorphism $\varphi:B_1 \ra B_2$  preserving the normalised Lebesgue measures such that
\[\varphi_* 
  \left ( [\sum_{n \geq 0} \mu_1^{* n} * \delta_{\varphi(s)}]|_{B_1} \right ) 
  =
  [\sum_{n \geq 0} \mu_2^{* n} * \delta_s]|_{B_2}
\eqstop\]
for Lebesgue almost every $s \in B_2$
\end{corollary}
\begin{proof}
  Write $M = M_{\pi_1} \cong M_{\pi_2}$ and $A_i$, for $i \in \{1,2\}$.  Denote by $[\nu_i] = \int [\sum_{n \geq 0} \mu_i^{* n} * \delta_s] \, \rmd \lambda(s)$ the maximal spectral type of $\bim{A_i}{\Ltwo(M)}{A_i}$ according to Proposition \ref{prop:disintegration-spectral-measure-FBCP}.  By Corollary \ref{cor:isomorphic-bimodules}, there are projections $p_1 \in A_1$, $p_2 \in A_2$ and an isomorphism $\phi:A_1p_1 \ra A_2p_2$ such that the bimodules $\bim{A_1p_1}{(p_1\Ltwo(M)p_1)}{A_1p_1}$ and $\bim{\phi(A_1p_1)}{(p_2\Ltwo(M)p_2)}{\phi(A_1p_1)}$ are isomorphic.  The projections $p_i$ are indicator functions of Lebesgue non-negligible Borel sets $B_i \subset \rS^1$ and the isomorphism $\phi$ equals $\varphi_*$ for some Borel isomorphism $\varphi: B_1 \ra B_2$ preserving the normalised Lebesgue measures.  Since the bimodules $\bim{A_1p_1}{(p_1\Ltwo(M)p_1)}{A_1p_1}$ and $\bim{A_2p_2}{(p_2\Ltwo(M)p_2)}{A_2p_2}$ are isomorphic via $\phi$, their maximal spectral types are isomorphic via $\varphi \times \varphi$.  Using their integral decomposition with respect to the projection on the first component of $\rS^1 \times \rS^1$ as it is calculated in Proposition \ref{prop:disintegration-spectral-measure-FBCP}, we obtain
\begin{align*}
  \left ( \int_{B_2} [\sum_{n \geq 0} \mu_2^{* n} * \delta_s]|_{B_2} \, \rmd \lambda(s) \right )
  & =
  (\varphi \times \varphi)_* \left ( \int_{B_1} [\sum_{n \geq 0} \mu_1^{* n} * \delta_s]|_{B_1} \, \rmd \lambda(s) \right ) \\
  & =
  (\varphi \times \id)_* \left ( \int_{B_2} [\sum_{n \geq 0} \mu_1^{* n} * \delta_{\varphi(s)}]|_{B_1} \, \rmd \lambda(s) \right ) \\
  & =
  \left ( \int_{B_2} \varphi_*([\sum_{n \geq 0} \mu_1^{* n} * \delta_{\varphi(s)}]|_{B_1}) \, \rmd \lambda(s) \right )
\end{align*}
As a result, for almost every $s \in B_2$, we obtain the equality
\[\varphi_* \left ( [\sum_{n \geq 0} \mu_1^{* n} * \delta_{\varphi(s)}]|_{B_1} \right ) = [\sum_{n \geq 0} \mu_2^{* n} * \delta_s]|_{B_2} \eqstop \]
\end{proof}

The next theorem follows by applying the previous one to some special cases.

\begin{theorem}
\label{thm:rigidity-result}
  No free Bogoljubov crossed product associated with a representation in the following classes is isomorphic to a free Bogoljubov crossed product associated with a representation in the other classes.
  \begin{enumerate}
  \item The class of representations $\lambda \oplus \pi_{\mathrm{ap}}$, where $\lambda$ is a multiple of the left regular representation of $\ZZ$ and $\pi_{\mathrm{ap}}$ is a faithful almost periodic representation of dimension at least $2$.
  \item The class of representations $\lambda \oplus \pi_{\mathrm{ap}}$, where $\lambda$ is a multiple of the left regular representation of $\ZZ$ and $\pi_{\mathrm{ap}}$ is a non-faithful almost periodic representation of dimension at least $2$.
  \item The class of representations $\rho \oplus \pi_{\mathrm{ap}}$, where $\rho$ is a representations of $\ZZ$ by multiplication with $\id_{\rS^1}$ on $\Ltwo_\RR(\rS^1, \mu)$, $\mu$ is a probability measure on $\rS^1$ such that $\mu^{*n}$ is singular for all $n$ and $\pi_{\mathrm{ap}}$ is a faithful almost periodic representation of dimension at least $2$.
  \item The class of representations $\rho \oplus \pi_{\mathrm{ap}}$, where $\rho$ is a representations of $\ZZ$ by multiplication with $\id_{\rS^1}$ on $\Ltwo_\RR(\rS^1, \mu)$, $\mu$ is a probability measure on $\rS^1$ such that $\mu^{*n}$ is singular for all $n$ and $\pi_{\mathrm{ap}}$ is a non-faithful almost periodic representation of dimension at least $2$.
  \item Faithful almost periodic representations of dimension at least $2$.
  \item Non-faithful, almost periodic representations of dimension at least $2$.
  \item \label{it:rigidiy:solid} The class of representations $\rho \oplus \pi$, where $\rho$ is mixing and $\dim \pi \leq 1$.
  \end{enumerate}
\end{theorem}
Note that by \cite{houdayershlyakhtenko11}, there are measures as mentioned item (iii) and (iv).

\begin{proof}
By Theorem \ref{thm:strong-solidity}, all free Bogoljubov crossed products associated with representations in \ref{it:rigidiy:solid} are strongly solid, but for all other free Bogoljubov crossed products $A \subset M$ is an amenable diffuse von Neumann subalgebra with a non-amenable normaliser. 

It remains to consider representations in (i) to (vi).  They satisfy the requirements of Corollaries \ref{cor:isomorphic-bimodules} and \ref{cor:measure-class-equal}.

We first claim that representations from (i) to (vi) with a faithful and non-faithful almost periodic part, respectively, cannot give rise to isomorphic free Bogoljubov crossed products.  Let $\pi$ be an orthogonal representation of $\ZZ$ and let $B \subset \rS^1$ be Lebesgue non-negligible.  The subgroup generated by the eigenvalues of the complexification of $\pi$ is dense if and only if the almost periodic part of $\pi$ is faithful.  So by Section \ref{sec:measure-from-bimodule}, the atoms of the spectral invariant of $\bim{pA_\pi}{p\Ltwo(M)p}{pA_\pi}$ are an ergodic equivalence relation on $B \times B$ if and only if $\pi$ has a faithful almost periodic part.  So Corollary \ref{cor:isomorphic-bimodules} proves our claim.

Let us now consider the weakly mixing part of the representations in the theorem. It is known that the spectral measure of the left regular representation of $\ZZ$ on $\ltwo_\RR(\ZZ)$ is the Lebesgue measure.  So from Corollary \ref{cor:measure-class-equal}, it follows that the representations whose weakly mixing part is the left regular representation, cannot give a free Bogoljubov crossed product isomorphic to a free Bogoljubov crossed product associated with any of the other representations in the theorem.  Finally, note that for any non-zero projection $p \in A_\pi$ the bimodules $\bim{pA_\pi}{\Ltwo(pM_\pi p)}{p A_\pi}$ is a direct sum of finite index bimodules if and only if the representation $\pi$ has no weakly mixing part.  So appealing to Corollary \ref{cor:isomorphic-bimodules}, we finish the proof.
\end{proof}


\bibliographystyle{mybibtexstyle}
\bibliography{operatoralgebras}

\begin{thebibliography}{CFW81}\setlength{\itemsep}{-1mm}\setlength{\parsep}{0m%
m}\small

\bibitem[BV12]{berbecvaes12}
M. Berbec and S. Vaes.
\newblock {W$^*$-superrigidity for group von Neumann algebras of left-right
  wreath products}.
\newblock arXiv:1210.0336.

\bibitem[CFW81]{connesfeldmanweiss81}
A. Connes, J. Feldman, and B. Weiss.
\newblock {An amenable equivalence relation is generated by a single
  transformation.}
\newblock {\em Ergodic Theory Dyn. Syst.} \textbf{1}, 431--450, 1981.

\bibitem[Con76]{connes76}
A. Connes.
\newblock {Classification of injective factors. Cases II$_1$, II$_\infty$,
  III$_\lambda$, $\lambda\neq 1$.}
\newblock {\em Ann. Math. (2)} \textbf{104}, 73--115, 1976.

\bibitem[Con94]{connes94}
A. Connes.
\newblock {\em {Noncommutative geometry.}}
\newblock San Diego, CA: Academic Press, 1994.

\bibitem[DR11]{dykemaredelmeier11}
K.~J. Dykema and D. Redelmeier.
\newblock {The amalgamated free product of hyperfinite von Neumann algebras
  over finite dimensional subalgebras.}
\newblock To appear in Houston J. Math.

\bibitem[Dyk92a]{dykema93-hyperfinite}
K.~J. Dykema.
\newblock {Free products of hyperfinite von Neumann algebras and free
  dimension.}
\newblock {\em Duke Math. J.} \textbf{69} (1), 97--119, 1993.

\bibitem[Dyk92b]{dykema93-matrix-model}
K.~J. Dykema.
\newblock {On certain free product factors via an extended matrix model.}
\newblock {\em J. Funct. Anal.} \textbf{112} (1), 31--60, 1993.

\bibitem[Dyk92c]{dykema94-interpolated}
K.~J. Dykema.
\newblock {Interpolated free group factors.}
\newblock {\em Pac. J. Math.} \textbf{163} (1), 123--135, 1994.

\bibitem[Dyk94]{dykema95-multi-matrix}
K.~J. Dykema.
\newblock {Amalgamated free products of multi-matrix algebras and a
  construction of subfactors of a free group factor.}
\newblock {\em Am. J. Math.} \textbf{117} (6), 1555--1602, 1995.

\bibitem[Dyk09]{dykema11-finite-vnalg}
K.~J. Dykema.
\newblock {A description of amalgamated free products of finite von Neumann
  algebras over finite-dimensional subalgebras.}
\newblock {\em Bull. Lond. Math. Soc.} \textbf{43} (1), 63--74, 2011.

\bibitem[Hou08]{houdayer10}
C. Houdayer.
\newblock {Structural results for free Araki-Woods factors and their continuous
  cores.}
\newblock {\em J. Inst. Math. Jussieu} \textbf{9} (4), 741--767, 2010.

\bibitem[Hou09]{houdayer10-non-free-gropup}
C. Houdayer.
\newblock {Strongly solid group factors which are not interpolated free group
  factors.}
\newblock {\em Math. Ann.} \textbf{346} (4), 969--989, 2010.

\bibitem[Hou12a]{houdayer12}
C. Houdayer.
\newblock {A class of II${_1}$ factors with an exotic abelian maximal amenable
  subalgebra}.
\newblock To appear in Trans. Amer. Math. Soc.

\bibitem[Hou12b]{houdayer12-structure}
C. Houdayer.
\newblock {Structure of II$_1$ factors arising from free Bogoljubov actions of
  arbitrary groups.}
\newblock arXiv:1209.5209.

\bibitem[HR10]{houdayerricard11-araki-woods}
C. Houdayer and {\'E}. Ricard.
\newblock {Approximation properties and absence of Cartan subalgebra for free
  Araki-Woods factors.}
\newblock {\em Adv. Math.} \textbf{228} (2), 764--802, 2011.

\bibitem[HS09]{houdayershlyakhtenko11}
C. Houdayer and D. Shlyakhtenko.
\newblock {Strongly solid II$_{1}$ factors with an exotic MASA.}
\newblock {\em Int. Math. Res. Not.} \textbf{2011} (6), 1352--1380, 2009.

\bibitem[HV12]{houdayervaes12}
C. Houdayer and S. Vaes.
\newblock {Type III factors with unique Cartan decomposition}.
\newblock arXiv:1203.1254v2.

\bibitem[Ioa12]{ioana12}
A. Ioana.
\newblock {Cartan subalgebras of amalgamated free product II$_1$ factors}.
\newblock arXiv:1207.0054v1.

\bibitem[IPP05]{ioanapetersonpopa08}
A. Ioana, J. Peterson, and S. Popa.
\newblock {Amalgamated free products of weakly rigid factors and calculation of
  their symmetry groups.}
\newblock {\em Acta Math.} \textbf{200} (1), 85--153, 2008.

\bibitem[IPV10]{ioanapopavaes10}
A. Ioana, S. Popa, and S. Vaes.
\newblock {A class of superrigid group von Neumann algebras.}
\newblock To appear in Ann. of Math. (2).

\bibitem[NS02]{neshveyevstormer02}
S. Neshveyev and E. St{\o}rmer.
\newblock {Ergodic theory and maximal abelian subalgebras of the hyperfinite
  factor.}
\newblock {\em J. Funct. Anal.} \textbf{195} (2), 239--261, 2002.

\bibitem[OP10a]{ozawapopa10-cartan1}
N. Ozawa and S. Popa.
\newblock {On a class of II$_1$ factors with at most one Cartan subalgebra.}
\newblock {\em Ann. Math. (2)} \textbf{172} (1), 713--749, 2010.

\bibitem[OP10b]{ozawapopa10-cartan2}
N. Ozawa and S. Popa.
\newblock {On a class of II$_1$ factors with at most one Cartan subalgebra.
  II.}
\newblock {\em Am. J. Math.} \textbf{132} (3), 841--866, 2010.

\bibitem[OW80]{ornsteinweiss80}
D.~S. Ornstein and B. Weiss.
\newblock {Ergodic theory of amenable group actions.}
\newblock {\em Bull. Am. Math. Soc., New Ser.} \textbf{2}, 161--164, 1980.

\bibitem[Oza04]{ozawa04-solid}
N. Ozawa.
\newblock {Solid von Neumann algebras}.
\newblock {\em Acta Math.} \textbf{192} (1), 111--117, 2004.

\bibitem[Pop93]{popa93}
S. Popa.
\newblock {Markov traces on universal Jones algebras and subfactors of finite
  index.}
\newblock {\em Invent. Math.} \textbf{111} (2), 375--405, 1993.

\bibitem[Pop01]{popa06-non-commutative-bernoulli-shifts}
S. Popa.
\newblock {Some rigidity results for non-commutative Bernoulli shifts.}
\newblock {\em J. Funct. Anal.} \textbf{230} (2), 273--328, 2006.

\bibitem[Pop02]{popa06}
S. Popa.
\newblock {On a class of type II$_1$ factors with Betti numbers invariants.}
\newblock {\em Ann. Math. (2)} \textbf{163} (3), 809--899, 2006.

\bibitem[Pop03]{popa06_2}
S. Popa.
\newblock {Strong rigidity of II$_1$ factors arising from malleable actions of
  $w$-rigid groups. I.}
\newblock {\em Invent. Math.} \textbf{165} (2), 369--408, 2006.

\bibitem[Pop04]{popa06_3}
S. Popa.
\newblock {Strong rigidity of II$_1$ factors arising from malleable actions of
  $w$-rigid groups. II.}
\newblock {\em Invent. Math.} \textbf{165} (2), 409--451, 2006.

\bibitem[Pop06a]{popa07-on-ozawa}
S. Popa.
\newblock {On Ozawa's property for free group factors.}
\newblock {\em Int. Math. Res. Not.} \textbf{2007} (11), 10 p., 2007.

\bibitem[Pop06b]{popa08-spectral-gap}
S. Popa.
\newblock {On the superrigidity of malleable actions with spectral gap.}
\newblock {\em J. Am. Math. Soc.} \textbf{21} (4), 981--1000, 2008.

\bibitem[PV09]{popavaes10-superrigidity}
S. Popa and S. Vaes.
\newblock {Group measure space decomposition of II$_1$ factors and
  W$^{*}$-superrigidity.}
\newblock {\em Invent. Math.} \textbf{182} (2), 371--417, 2010.

\bibitem[PV11]{popavaes11_2}
S. Popa and S. Vaes.
\newblock {Unique Cartan decomposition for II$_1$ factors arising from
  arbitrary actions of free groups}.
\newblock arXiv:1111.6951v2.

\bibitem[PV12]{popavaes12}
S. Popa and S. Vaes.
\newblock {Unique Cartan decomposition for II$_1$ factors arising from
  arbitrary actions of hyperbolic groups}.
\newblock To appear in J. Reine Angew. Math.

\bibitem[R{\u{a}}d94]{radulescu94-interpolated}
F. R{\u{a}}dulescu.
\newblock {Random matrices, amalgamated free products and subfactors of the von
  Neumann algebra of a free group, of noninteger index.}
\newblock {\em Invent. Math.} \textbf{115} (2), 347--389, 1994.

\bibitem[Shl97]{shlyakhtenko97}
D. Shlyakhtenko.
\newblock {Free quasi-free states.}
\newblock {\em Pac. J. Math.} \textbf{177} (2), 329--368, 1997.

\bibitem[Shl98]{shlyakhtenko98-applications}
D. Shlyakhtenko.
\newblock {Some applications of freeness with amalgamation.}
\newblock {\em J. Reine Angew. Math.} \textbf{500}, 191--212, 1998.

\bibitem[Shl99]{shlyakhtenko99}
D. Shlyakhtenko.
\newblock {$A$-valued semicircular systems.}
\newblock {\em J. Funct. Anal.} \textbf{166} (1), 1--47, 1999.

\bibitem[Shl03]{shlyakhtenko03-multiplicity}
D. Shlyakhtenko.
\newblock {On multiplicity and free absorbtion of free Araki-Woods factors.}
\newblock arXiv:math/0302217v1, 2003.

\bibitem[Sin55]{singer55}
I.~M. Singer.
\newblock {Automorphisms of finite factors.}
\newblock {\em Am. J. Math.} \textbf{77}, 117--133, 1955.

\bibitem[Spe98]{speicher98}
R. Speicher.
\newblock {Combinatorial theory of the free product with amalgamation and
  operator-valued free probability theory.}
\newblock {\em Mem. Am. Math. Soc.} \textbf{132} (627), 1998.

\bibitem[SW82]{schmidtwalters82}
K. Schmidt and P. Walters.
\newblock {Mildly mixing actions of locally compact groups.}
\newblock {\em Proc. Lond. Math. Soc., III. Ser.} \textbf{45}, 506--518, 1982.

\bibitem[Ued99]{ueda99}
Y. Ueda.
\newblock {Amalgamated free product over Cartan subalgebra.}
\newblock {\em Pac. J. Math.} \textbf{191} (2), 359--392, 1999.

\bibitem[Ued00]{ueda03}
Y. Ueda.
\newblock {Fullness, Connes' $\chi$-groups, and ultra-products of amalgamated
  free products over Cartan subalgebras.}
\newblock {\em Trans. Am. Math. Soc.} \textbf{355} (1), 349--371, 2003.

\bibitem[Ued02]{ueda04}
Y. Ueda.
\newblock {Amalgamated free product over Cartan subalgebra. II: Supplementary
  results \& examples.}
\newblock In H. Kosaki, editor, {\em Operator algebras and applications.
  Proceedings of the US-Japan seminar held at Kyushu University, Fukuoka,
  Japan, June 7--11, 1999}, volume~38 of {\em Adv. Stud. Pure Math.}, pages
  239--265. {Tokyo: Mathematical Society of Japan}, 2004.

\bibitem[Ued12]{ueda12}
Y. Ueda.
\newblock {Some analysis on amalgamated free products of von Neumann algebras
  in non-tracial setting.}
\newblock arXiv:1203.1806v2.

\bibitem[Vae06a]{vaes07}
S. Vaes.
\newblock {Rigidity results for Bernoulli actions and their von Neumann
  algebras (after Sorin Popa).}
\newblock In {\em S{\'e}minaire Bourbaki. 2005/2006}, volume 311 of {\em
  Ast{\'e}risque}, pages 237--294. Paris: Soci{\'e}t{\'e} Math{\'e}matique de
  France, 2007.

\bibitem[Vae06b]{vaes09}
S. Vaes.
\newblock {Factors of type II$_1$ without non-trivial finite index subfactors.}
\newblock {\em Trans. Am. Math. Soc.} \textbf{361} (5), 2587--2606, 2009.

\bibitem[VDN92]{voiculescudykemanica92}
D.~V. Voiculescu, K. Dykema, and A. Nica.
\newblock {\em {Free random variables. A noncommutative probability approach to
  free products with applications to random matrices, operator algebras and
  harmonic analysis on free groups.}}
\newblock CRM Monograph Series. 1. Providence, RI: American Mathematical
  Society, 1992.

\bibitem[Voi85]{voiculescu85}
D. Voiculescu.
\newblock {Symmetries of some reduced free product C$^*$-algebras.}
\newblock In {\em Operator algebras and their connections with topology and
  ergodic theory, (Buşteni/Rom. 1983)}, volume 1132 of {\em Lect. Notes
  Math.}, pages 556 -- 588. {Berlin: Springer}, 1985.

\bibitem[Voi95]{voiculescu95}
D.~V. Voiculescu.
\newblock {Operations on certain non-commutative operator-valued random
  variables.}
\newblock In A. Connes et~al., editors, {\em {Recent advances in operator
  algebras. Collection of talks given in the conference on operator algebras
  held in Orl\'eans, France in July 1992}}, volume 232 of {\em Ast{\'e}risque},
  pages 243--275. Paris: Soci{\'e}t{\'e} Math{\'e}matique de France, 1995.

\end{thebibliography}

{\small \parbox[t]{150pt}{
    Sven Raum \\
    Department of Mathematics\\
    K.U.Leuven, Celestijnenlaan 200B \\
    B--3001 Leuven \\
    Belgium \\
    {\footnotesize sven.raum@wis.kuleuven.be}}}

\end{document}